\documentclass[a4paper,12pt,reqno]{amsart}

\usepackage[utf8]{inputenc}
\usepackage[T1]{fontenc}
\usepackage{indentfirst}
\usepackage[english]{babel}
\usepackage{csquotes}
\usepackage{amsmath,amsfonts,amssymb,amsopn,amscd,amsthm}
\usepackage{mathrsfs,shuffle,bbm}
\usepackage{stmaryrd}
\usepackage{graphicx}
\usepackage{mathtools}
\usepackage[dvipsnames]{xcolor}
\usepackage[colorlinks=true,citecolor=OrangeRed,linkcolor=NavyBlue]{hyperref}
\usepackage{tikz}
\usepackage{enumitem}
\usepackage[a4paper,twoside,margin=0.8in]{geometry}

\newtheorem{theo}{Theorem}[section]
\newtheorem{defn}[theo]{Definition}
\newtheorem{exa}[theo]{Example}
\newtheorem{prop}[theo]{Proposition}
\newtheorem{lemma}[theo]{Lemma}

\theoremstyle{remark}
\newtheorem{remark}[theo]{Remark}

\makeatletter
\def\thm@space@setup{\thm@preskip=0.5cm \thm@postskip=0.5cm}
\makeatother
\newcommand{\implication}[2]{%
  \mbox{$\text{#1}\implies\text{#2}$}.}%

\newcommand{\E}{\mathrm{e}}
\newcommand{\I}{\mathrm{i}}
\newcommand{\R}{\mathbb{R}}
\newcommand{\N}{\mathbb{N}}
\newcommand{\Z}{\mathbb{Z}}
\newcommand{\Tor}{\mathbb{T}}
\newcommand{\esper}{\mathbb{E}}
\newcommand{\proba}{\mathbb{P}}
\newcommand{\ibar}{\bar{\imath}}
\newcommand{\jbar}{\bar{\jmath}}
\newcommand{\PLM}[1]{#1}
\newcommand{\var}{\mathrm{var}}
\newcommand{\spa}{\mathscr{X}}
\newcommand{\spaex}{(\mathcal{X},d,\mu)}
\newcommand{\Ccal}{\mathscr{C}}
\newcommand{\Mcal}{\mathscr{M}}

\setlist[enumerate]{itemsep=10pt,topsep=10pt}
\setlist[itemize]{itemsep=5pt,topsep=5pt}

\title{Fluctuations of the Gromov--Prohorov sample model}
\author{Jacques de Catelan}
\address{\noindent Université Paris-Sud, Faculté des Sciences d’Orsay --- Institut de mathématiques d'Orsay, Bâtiment 307 ---
F-91405 Orsay, France.}
\email{jacques.de-catelan@math.u-psud.fr}

\author{Pierre-Loïc Méliot}
\address{\noindent Université Paris-Sud, Faculté des Sciences d’Orsay --- Institut de mathématiques d'Orsay, Bâtiment 307 ---
F-91405 Orsay, France.}
\email{pierre-loic.meliot@math.u-psud.fr}

\begin{document}
\date{}

\begin{abstract}
In this paper, we study the fluctuations of observables of metric measure spaces which are random discrete approximations $\spa_n$ of a fixed arbitrary (complete, separable) metric measure space $\spa=\spaex$. These observables $\Phi(\spa_n)$ are polynomials in the sense of Greven--Pfaffelhuber--Winter, and we show that for a generic model space $\spa$, they yield asymptotically normal random variables. However, if $\spa$ is a compact homogeneous space, then the fluctuations of the observables are much smaller, and after an adequate rescaling, they converge towards probability distributions which are not Gaussian. Conversely, we prove that if all the fluctuations of the observables $\Phi(\spa_n)$ are smaller than in the generic case, then the measure metric space $\spa$ is compact homogeneous. The proofs of these results rely on the Gromov reconstruction principle, and on an adaptation of the method of cumulants and mod-Gaussian convergence developed by Féray--Méliot--Nikeghbali. As an application of our results, we construct a statistical test of the hypothesis of symmetry of a compact Riemannian manifold.
\end{abstract}

\maketitle
\tableofcontents

\section{Introduction}
Let $ \spa = \spaex$ be a metric space which we assume to be complete, separable and equipped with a probability measure $\mu$ over the Borel algebra of $\mathcal{X}$; and $(X_n)_{n \in \mathbb{N}}$ be a sequence of independent random variables with the same law $\mu$. We study here the approximation of $\spa = \spaex$ by the random discrete metric space
\begin{align*}
\spa_n = \left( \mathcal{X}_n = \{X_1,\ldots,X_n\}, d, \frac{1}{n}\sum_{i=1}^n \delta_{X_i}\right)
\end{align*}
for the Gromov-weak topology; we call this discrete approximation the \emph{Gromov--Prohorov random sample model}. The Gromov-weak topology is based on the idea that a sequence of metric measure spaces converges if and only if all finite subspaces sampled from these spaces converge. This is formalized by using real-valued observables called \emph{polynomials} and introduced by Greven, Pfaffelhuber and Winter in \cite{greven2009convergence}: they are the functions $\Phi$ defined by
$$\Phi(\spaex) = \int_{\mathcal{X}^p} \varphi((d(x_i,x_j))_{1\leq i<j\leq p})\,\mu(dx_1)\cdots \mu(dx_p),$$

\noindent where $\varphi : \R^{\binom{p}{2}} \to \R$ is an arbitrary continuous bounded function. By using the theorem of convergence of empirical measures (see \cite[Theorem 3]{Var58}), one proves readily the almost sure convergence of $\spa_n$ toward $\spa$ (see Theorem \ref{theo:as_convergence}). In this paper, we will study the fluctuations of the polynomials $\Phi(\spa_n)$ with respect to their limits $\Phi(\spa)$. The evaluation of a polynomial $\Phi$ on the space $\spa_n$ is a sum of dependent random variables
\begin{align*}
\Phi(\spa_n) = \frac{1}{n^p} \,\sum_{\ibar \in [\![1,n ]\!]^p } \varphi(d(X_{\ibar})),
\end{align*} 
where we abbreviate $\varphi(d(X_{\ibar})) := \varphi((d(X_{i_a},X_{i_b}))_{1\leq a ,b \leq p})$ for a sequence of indices $\ibar=(i_1,\ldots,i_p)$. This dependency between the random variables is sparse: if we associate to these variables a graph describing the dependency between those variables, then when $n$ goes to infinity the maximal degree of a vertex of this graph becomes negligible against the number of vertices (variables). 
\PLM{This sparse dependency leads to central limit theorems, but the limiting distribution is not necessarily Gaussian, and it depends on the size of the variance of $\Phi(\spa_n)$, for which there are two cases.}
\bigskip

We shall see that the variance $\var(S_n(\varphi,\spa))$ with $S_n(\varphi,\spa) = n^p\, \Phi(\spa_n) $ is a polynomial in the variable $n$ with coefficients depending on the function $\varphi$ and the space $ \spa$; this variance is at most of order $n^{2p-1}$ and therefore, $\var(\Phi(\spa_n))$ is of order at most $1/n$. 
\begin{itemize}
	\item In a first part, we study the case where the variance of $\Phi(\spa_n)$ is of order exactly $1/n$. We call this setting the \emph{generic case}, and it corresponds to fluctuations which are asymptotically normal.  We study the combinatorics of the cumulants of the variable $S_n(\varphi,\spa)$ by using the theory of dependency graphs and mod-Gaussian convergence developed recently by Féray, Méliot and Nikeghbali (see \cite{feray2016mod}); and we prove the mod-Gaussian convergence of the sequence $S_n(\varphi,\spa)$ adequately renormalized. This leads to a central limit theorem for the variables 
	$$Y_n(\varphi,\spa) = \frac{\Phi(\spa_n) - \esper[\Phi(\spa_n)]}{\sqrt{\var(\Phi(\spa_n))}};$$ 
	the limiting distribution is the standard Gaussian distribution, and we also obtain the normality zone of this approximation, moderate deviation estimates and a Berry--Esseen inequality (Theorem \ref{theo:generic_case}). 
	In \cite{2017arXiv171206841F}, similar techniques were used in the study of the fluctuations of observables of random graph, random permutation and random integer partition models parametrised respectively by the space of graphons, the space of permutons and the Thoma simplex.
	\item In a second part, we study the case where the variance of $\Phi(\spa_n)$ is at most of order $1/n^2$ for any polynomial $\Phi$. We call this setting a \emph{globally singular point} $\spa$ of the Gromov--Prohorov sample model. It corresponds to the following condition: for any $p \geq 1$ and any $\varphi \in \Ccal_b(\R^{\binom{p}{2}})$,
\begin{align*}
 \sum_{1 \leq i,j \leq p} \mathrm{cov}\left(\varphi(d(X_1,\dots,X_i,\dots,X_p)), \varphi(d(X_1',\dots,X_{j-1}',X_i,X_{j+1}',\dots,X_p'))\right) = 0,
\end{align*}
where $(X_n')_{n \in \N}$ is an independent copy of $(X_n)_{n \in \N}$, and where in each summand the second vector contains all the variables $X_1',\ldots,X_p'$, except $X_j'$ which is replaced by $X_i$. This identity is difficult to analyse: therefore, we shall study the simpler case where each of the covariances in the sum vanishes. In particular, 
\begin{align*}
\mathrm{cov}\left(\varphi(d(X_1,X_2,\dots,X_p)), \varphi(d(X_1,X_2',\dots,X_p'))\right) = 0.
\end{align*}
It turns out that this second identity is equivalent to  $\spa$ being a compact homogeneous space (in the space of metric measure spaces); see Theorem \ref{theo:homogeneous}.
\PLM{We are thus able to relate a probabilistic condition to a geometric condition on the space; this result is a bit surprising, and for instance it ensures that when approximating an ellipse and a circle by the Gromov--Prohorov sample model, the convergence is much faster for the circle and does not have the same kind of asymptotic fluctuations.}
The proof of the equivalence relies notably on Gromov's reconstruction theorem \cite{gromov2007metric}. Now, in this situation, we cannot directly use the theory of mod-Gaussian convergence and dependency graphs in order to prove all the probabilistic results that we obtained in the generic case. However, by using the symmetry of the space, we are able to obtain for this singular case a better upper bound of the cumulants. It allows us to prove a central limit theorem for the random variables $Y_n(\varphi,\spa)$, but the limit is not necessarily the Gaussian distribution; see Theorem \ref{theo:singular_case}.
\end{itemize}   

\noindent The reader might wonder why we consider that replacing the hypothesis "the sums of covariances vanish" by "all the covariances vanish" is a reasonable restriction in the study of the singular models. In fact, we believe that the two conditions are equivalent; we shall say a short bit about this at the beginning of Section \ref{sec:homogeneous_case}, and we plan to address this question in forthcoming works.\bigskip

The theoretical results of this article lead to a better understanding of the possible behaviors of random variables stemming from a \emph{mod-Gaussian moduli space}; this kind of classifying object for random models has been introduced in \cite{2017arXiv171206841F}. Let us restate the previous discussion with this viewpoint. To any point $\spa$ of the space $\mathbb{M}$ of complete separable metric spaces endowed with a probability measures, one can associate a sequence of random models $(\spa_n)_{n \in \N}$ which are discrete approximations of $\spa$, and such that $\spa_n \to_{\proba} \spa$ as $n$ goes to infinity. Moreover, for a generic point $\spa$, an algebra of observables of spaces in $\mathbb{M}$ yields random variables $\Phi(\spa_n)$ such that $n^{-1/2}(\Phi(\spa_n)-\Phi(\spa))$ is always asymptotically normal. However, some special points $\spa \in \mathbb{M}$ yield observables such that $n^{-1/2}(\Phi(\spa_n)-\Phi(\spa))$ always goes to $0$ (in probability). The identification of these singular models $\spa$ is then a natural question, and for those models, one can be interested in the asymptotics of $\Phi(\spa_n)-\Phi(\spa)$ with a different rescaling (here, we shall look at $n^{-1}(\Phi(\spa_n)-\Phi(\spa))$). The exact same approach has been used in \cite{2017arXiv171206841F} with the space of graphons for models of random graphs, the space of permutons for models of random permutations, and the Thoma simplex for models of random integer partitions. Until now, we believed that the singular points of a mod-Gaussian moduli space still yielded observables which were asymptotically normal, albeit with a different rescaling. Indeed, this is what happens for singular graphons (Erd\H{o}s--Rényi random graphs) and for singular models of random integer partitions (Plancherel and Schur--Weyl measures). However this is not a general phenomenon: with the Gromov--Prohorov sample model, we encounter the first known example where singular points yield observables which \emph{are not} asymptotically normal after appropriate rescaling.
\bigskip

An application of our identification of the singular points of the space $\mathbb{M}$ of measured metric spaces is a procedure of statistical testing of the hypothesis of symmetry of a manifold. Suppose given a compact Riemannian manifold $\mathcal{X}$, endowed with its geodesic distance $d$ and with the unique probability measure $\mu$ which is proportional to the volume form of the manifold. We want to know whether $\mathcal{X}$ is a compact homogeneous space (see Theorem \ref{theo:homogeneous} for precise definitions). For instance, assume that one is given a surface homeomorphic to the real sphere $\mathbb{S}^2$ and endowed with a Riemannian structure; one wants to decide whether this structure is the canonical structure of symmetric space $\mathbb{S}^2 = \mathrm{SO}(3)/\mathrm{SO}(2)$. One can observe the manifold $\mathcal{X}$ as follows:
\begin{itemize}
	\item one can take independent random points $x_i$ on $\mathcal{X}$ according to $\mu$;
	\item one can measure the distances $d(x_i,x_j)$ between the observed points.
\end{itemize}
Fix a polynomial $\Phi$ as defined previously. If the triple $\spa =(\mathcal{X},d,\mu)$ is truly a homogeneous space, then the fluctuations of $\Phi(\spa_n)$ tend to be small, of order $n^{-1}$. Therefore, given a $2n$ sample of points $(x_{i},x_{i}')_{1\leq i \leq n}$ and a large threshold $t_\alpha$, if $\spa_{n}$ (respectively, $\spa_n'$) denotes the approximation of $\spa$ constructed from the family of points $\{x_{1},x_{2},\ldots,x_{n}\}$ (respectively, $\{x_1',x_2',\ldots,x_n'\}$), then
$$Z_{n} = n\,|\Phi(\spa_n)-\Phi(\spa_n')|$$
should be smaller than $t_\alpha$ with large probability $1-\alpha$ (for $n$ large). On the contrary, if $(\mathcal{X},d,\mu)$ is not homogeneous, then the fluctuations of $\Phi(\spa_{n})$ are generically of order $\frac{1}{\sqrt{n}}$, so one expects $Z_{n}$ to be larger than $t_\alpha$ with large probability (again, for $n$ large). We make this argument precise at the end of our paper, by describing in details the procedure of statistical hypothesis testing for the symmetry of $\spa$.
\bigskip

\noindent \textbf{Outline of the article.} The paper is organized as follows. In Section \ref{sec:metric_measure_spaces}, we will recall some definitions and facts about metric measure spaces. Section \ref{sec:method_cumulants} introduces the method of cumulants, the theory of dependency graphs and all the probabilistic results that we can obtain from this method. In Section \ref{sec:generic_fluctuations}, we apply this theory to the generic case of the random sample model to get several probabilistic results about the model including a central limit theorem, the normality zone, moderate deviations and a Berry--Esseen bound for the random variables $Y_n(\varphi,\spa)$.\medskip

Section \ref{sec:homogeneous_case} details the singular case, and we prove the equivalence between having a small variance for the model, and $\spa$ being a compact homogeneous space. We obtain also in this case a finer bound on the cumulants, a non-Gaussian central limit theorem for the observables $\Phi(\spa_n)$, and concentration inequalities for these random variables. In Section \ref{sec:circle}, we provide an explicit counterexample for the asymptotic normality of observables of the sample model of an homogeneous space. This section also enables us to explain in more details the combinatorics of moments and cumulants of the polynomial observables, and how to compute them concretely. Finally, Section \ref{sec:statistics} is devoted to the description of the statistical test for symmetry that has been briefly presented above.
\bigskip

\section{Metric measure spaces}\label{sec:metric_measure_spaces}

In this section, we recall the theory of metric measure spaces and of the Gromov–Prohorov topology, following very closely \cite[Section 2]{greven2009convergence}.

\subsection{Definitions}
For any topological space $\mathcal{X}$, we denote $\Ccal_b(\mathcal{X}) $ the set of continuous bounded functions $\mathcal{X} \to \mathbb{R}$; $\Ccal(\mathcal{X}) $ the set of continuous functions $\mathcal{X} \to \mathbb{R}$; $\mathscr{B}(\mathcal{X})$ the set of Borel subsets of $\mathcal{X}$; and $\Mcal^1(\mathcal{X})$ the set of Borel probability measures over $\mathcal{X}$. A measurable map $f : \mathcal{X} \to \mathcal{Y}$ between two topological spaces induces a map $f_* : \Mcal^1(\mathcal{X}) \to \Mcal^1(\mathcal{Y})$ (push-forward of measures): for any Borel subset $A \subset \mathcal{Y}$, $(f_*\mu)(A) = \mu(f^{-1}(A))$.

\begin{defn}
A {\normalfont metric measure space} is a complete and separable metric space $(\mathcal{X},d)$ which is endowed with a probability measure $\mu \in \Mcal^1(\mathcal{X})$. We say that two metric measure spaces $\spaex$ and $(\mathcal{X}',d',\mu')$ are measure-preserving isometric if there exists an isometry $\psi$ between the supports of $\mu$ on $(\mathcal{X},d)$ and of $\mu'$ on $(\mathcal{X}',d')$, such that $\mu' = \psi_* \mu$.  
\end{defn}

We denote $\mathbb{M}$ the space of metric measure spaces (in short, mm-spaces) modulo measure-preserving isometries. In the sequel, unless explicitly stated, given a mm-space $\spaex$, we will always suppose that the space $\mathcal{X}$ is exactly the support of the measure $\mu$.
Let $\spa = \spaex \in \mathbb{M}$ and 
\begin{align*}
\mathbb{R}^{\mathrm{met}} \vcentcolon= \{(d_{i,j})_{1 \leq i < j < \infty} \,\,|\,\,  \forall 1 \leq i < j < k < \infty,\,\,d_{i,j} + d_{j,k} \geq d_{i,k} \}.
\end{align*}
the space of infinite pseudo-distance matrices. We introduce the maps:
\begin{align*}
\iota^\spa \colon \mathcal{X}^\mathbb{N} &\to \mathbb{R}^{\mathrm{met}} \\
(x_n)_{n \in \mathbb{N}} &\mapsto (d(x_i,x_j))_{1 \leq i < j < \infty},
\end{align*}
and 
\begin{align*}
S \colon \mathcal{X} &\to (\mathcal{X}^\mathbb{N} \to \mathcal{X}^\mathbb{N}) \\
x &\mapsto \left(S^x \vcentcolon=  (x_n)_{n \in \mathbb{N}} \mapsto(x,x_0,x_1,x_2,\ldots)\right).
\end{align*}

\begin{defn}
We define the {\normalfont distance matrix distribution} of $\spa$ by
\begin{align*}
\nu^\spa \vcentcolon= (\iota^\spa)_* \mu^{\otimes \mathbb{N}},
\end{align*}
and the {\normalfont pointed distance matrix distribution} by
\begin{align*}
\nu \colon \mathcal{X} &\to \Mcal^1(\mathbb{R}^{\mathrm{met}})
\\
x &\mapsto \nu^x \vcentcolon= (\iota^\spa \circ S^x)_* \mu^{\otimes \mathbb{N}}.
\end{align*}
\end{defn}
The distance matrix distribution characterizes the metric measure space in $\mathbb{M}$. It means that if $\nu^{\spa_1} = \nu^{\spa_2}$, then $\spa_1 $ is measure-preserving isometric to $\spa_2$. This follows from Gromov's reconstruction theorem for metric measure spaces \cite[Paragraph $3 \frac{1}{2}.5$]{gromov2007metric}. 
\medskip

\subsection{Polynomials and the Gromov--Prohorov distance}
We associate to any bounded continuous map $ \varphi \in \Ccal_b(\mathbb{R}^{\binom{p}{2}})$ a map $\Phi = \Phi^{p,\varphi} \colon \mathbb{M} \to \mathbb{R}$ called a \emph{polynomial} on $\mathbb{M}$ and defined by
\begin{align*}
\Phi(\spa=\spaex) = \int_{\mathbb{R}^{\mathrm{met}}} \varphi((d_{i,j})_{1 \leq i < j \leq p})\,\nu^\spa((d_{i,j})_{1 \leq i < j \leq p}),
\end{align*}
We denote $\Pi$ the real algebra of polynomials on $\mathbb{M}$. Applying the definition of the distance-matrix distribution as a pushed-forward measure, we have
\begin{align*}
\Phi(\spaex) = \int_{\mathcal{X}^p} \varphi((d(x_i,x_j))_{1 \leq i < j \leq p})\,\mu^{\otimes p}(x_1,\dots,x_p).
\end{align*}
\begin{defn}
The Gromov-weak topology is the initial topology on $\mathbb{M}$ associated to the family of polynomials $(\Phi^{p,\varphi})_{p,\varphi}$. In the sequel we endow $\mathbb{M}$ with this topology.
\end{defn}

\medskip

The Gromov-weak topology can be metrized by the Gromov--Prohorov distance, where we optimally embed the two metric measure spaces into a common mm-space and then take the Prohorov distance between the image measures. Given $\mu$ and $\nu$ two probability measures on a metric space $(\mathcal{Z},d_{\mathcal{Z}})$, their Prohorov distance is
\begin{align*}
d_{\mathrm{Pr}}^{(\mathcal{Z},d_{\mathcal{Z}})} = \inf \{\epsilon > 0 \, | \, \forall A \in \mathscr{B}(\mathcal{Z}), \mu(A) \leq \nu(A^\epsilon)+\epsilon, \nu(A) \leq \mu(A^\epsilon)+\epsilon \},
\end{align*}
where $A^\epsilon=\{z \in \mathcal{Z}\,|\,d_{\mathcal{Z}}(z,A)<\epsilon\}$. It is well known to metrise the weak convergence of probability measures in $\mathscr{M}^1(\mathcal{Z})$ \cite[Theorem 6.8]{billing}.

\begin{defn}
The Gromov--Prohorov distance between two mm-spaces $\spa = (\mathcal{X},d_\mathcal{X},\mu_\mathcal{X})$ and $\mathscr{Y} = (\mathcal{Y},d_\mathcal{Y},\mu_\mathcal{Y})$ in $\mathbb{M}$ is defined by
\begin{align*}
d_{\mathrm{GPr}}(\spa,\mathscr{Y}) = \inf_{(\varphi_\mathcal{X},\varphi_\mathcal{Y},\mathcal{Z})} d_{\mathrm{Pr}}^{(\mathcal{Z},d_{\mathcal{Z}})}((\psi_\mathcal{X})_* \mu_\mathcal{X}, (\psi_\mathcal{Y})_* \mu_\mathcal{Y}),
\end{align*}
where the infimum is taken over all pairs of isometric embeddings $\psi_{\mathcal{X}}$ and $\psi_{\mathcal{Y}}$ from $(\mathcal{X},d_\mathcal{X})$ and $(\mathcal{Y},d_\mathcal{Y})$ into some common metric space $(\mathcal{Z},d_{\mathcal{Z}})$.
\end{defn}

\begin{theo}
Given a sequence of mm-spaces $(\spa_n=(\mathcal{X}_n,\mu_n,d_n))_{n \in \N}$ and another mm-space $\spa=(\mathcal{X},d,\mu)$ in $\mathbb{M}$, the following assertions are equivalent:
\begin{enumerate}
	\item The sequence $(\spa_n)_{n \in \N}$ converges to $\spa$ with respect to the Gromov--Prohorov distance.
	\item The sequence of distance matrix distributions $(\nu^{\spa_n})_{n \in \N}$ converges weakly to $\nu^\spa$. 
	\item The sequence $(\spa_n)_{n \in \N}$ converges to $\spa$ with respect to the Gromov--weak topology: for any polynomial $\Phi^{p,\varphi}$ associated to a bounded and continuous function $\varphi \in \mathscr{C}_b(\R^{\binom{p}{2}})$, we have $\Phi^{p,\varphi}(\spa_n) \to_{n \to \infty} \Phi^{p,\varphi}(\spa)$.
	\item For any $p \geq 2 $ and any compactly supported and continuous function $\varphi \in \mathscr{C}_c(\R^{\binom{p}{2}})$, we have $\Phi^{p,\varphi}(\spa_n) \to_{n \to \infty} \Phi^{p,\varphi}(\spa)$.
\end{enumerate}
Furthermore, the metric space $(\mathbb{M}, d_{\mathrm{GPr}})$ is complete and separable, so the space $\mathbb{M}$ is polish.
\end{theo}

\begin{proof}
The equivalence of the three first points and the polish character are respectively Theorems 5 and Theorem 1 in \cite[Theorem 5]{greven2009convergence}; we also refer to \cite{Lohr13} for further details on the Gromov--Prohorov metric. We have obviously (3) $\Rightarrow$ (4). Conversely, note that (4) amounts to the \emph{vague} convergence of the distance matrix distributions $\nu^{\spa_n}$ towards $\nu^{\spa}$. However, for probability measures, vague convergence and weak convergence are equivalent (the difference is that for vague convergence we can have a positive mass that escapes to infinity, but this does not happen if we specify the limit and if this limit is a probability measure); see \cite[Lemma 5.20]{Kallenberg02}. Therefore, (4) $\Rightarrow$ (2).
\end{proof}

\begin{remark}\label{remark:dense}
As a consequence of the fourth item in the theorem above and of the Stone--Weierstrass theorem, in order to control the Gromov-weak topology, we can use a \emph{countable} family $H$ of polynomials $(\Phi^{p,\varphi})_{p,\varphi}$  associated to functions $\varphi \colon \mathbb{R}^{\binom{p}{2}} \to \mathbb{R}$ with compact support.
\end{remark}

\subsection{Almost sure convergence of the sample model}\label{sub:almost_sure_convergence}
Let $ \spa = \spaex$ in $\mathbb{M}$ and $(X_n)_{n \in \mathbb{N}}$ be a sequence of random and independent variables with the same law $\mu$. We define
\begin{align*}
\spa_n = \left( \mathcal{X}_n = \{X_1,\dots,X_n\}, {d|}_{\mathcal{X}_n}, \mu_n = \frac{1}{n}\sum_{i=1}^n \delta_{X_i}\right).
\end{align*}
Then, taking $\Phi \in H$ (see Remark~\ref{remark:dense}), we have
\begin{align*}
\Phi(\spa_n) &=  \int_{\mathcal{X}^p} \varphi((d(x_i,x_j))_{1 \leq i < j \leq p})\,\mu_n^{\otimes p}(x_1,\dots,x_p)\\
&{\longrightarrow}_{n \to \infty} \int_{\mathcal{X}^p} \varphi((d(x_i,x_j))_{1 \leq i < j \leq p})\,\mu^{\otimes p}(x_1,\dots,x_p) = \Phi(\mathcal{X}).
\end{align*}
Indeed, $\mu_n$ converges almost surely to $\mu$ for the weak topology of probability measures (see for instance \cite{Var58}), so the same is true for $\mu^{\otimes p}_n$ toward $\mu^{\otimes p}$ (see \cite[Chapter 1, Example 3.2]{billing}). This implies the following theorem:

\begin{theo}\label{theo:as_convergence}
We have the almost sure convergence $\spa_n \underset{a.s.}{{\longrightarrow}} \spa$ in the space $\mathbb{M}$ of mm-spaces: 
$$\proba[\Phi(\spa_n) \to_{n \to \infty} \Phi(\spa)\text{ for any polynomial }\Phi \in \Pi]=1$$
or equivalently,
$$\proba[d_{\mathrm{GPr}}(\spa_n,\spa) \to_{n \to \infty} 0] =1.$$
\end{theo}

\PLM{We can also prove the theorem by using the Gromov--Prohorov distance; indeed, by choosing $\mathcal{Z}=\mathcal{X}$ as the common metric space in which one embeds $\mathcal{X}_n$ and $\mathcal{X}$, and the identity maps for the isometric embeddings, we see that
$$d_{\mathrm{GPr}}(\spa_n,\spa) \leq d_{\mathrm{Pr}}(\mu_n,\mu),$$
and the convergence to $0$ of the right-hand side is the Glivenko--Cantelli convergence of empirical measures. Estimates on the speed of convergence of $\esper[d_{\mathrm{Pr}}(\mu_n,\mu)]$ are given in \cite{Dud69}, but they depend strongly on the space $\spa$: if $k$ denotes the entropic dimension of $\spa=\spaex$, then in general one cannot prove a better bound than 
$\esper[d_{\mathrm{Pr}}(\mu_n,\mu)] = O(n^{-\frac{1}{k+2+\epsilon}})$;
see Theorem 4.1 in \cite{Dud69}. However, if instead of the Gromov--Prohorov distance one uses polynomial observables $\Phi$ in order to control the speed of convergence, then the results of this paper will prove that  essentially there are only two possible speeds of convergence:
\begin{itemize}
 	\item in the generic case, $|\Phi(\spa_n)-\Phi(\spa)| = O(n^{-\frac{1}{2}})$; more precisely, there exists a bilinear map
 	$$\kappa^2 : \Pi^2 \to \Pi$$
 	such that, for any polynomial $\Phi=\Phi^{p,\varphi}$, we have the convergence in law $$\frac{\Phi(\spa_n)-\Phi(\spa)}{n^{1/2}}\rightharpoonup_{n \to \infty} \mathcal{N}(0,\kappa^2(\Phi)(\spa));$$
 	see Theorem \ref{theo:generic_case}.
 	\item in the case of compact homogeneous spaces, $|\Phi(\spa_n)-\Phi(\spa)| = O(n^{-1})$; more precisely, for any polynomial $\Phi=\Phi^{p,\varphi}$, there exists a random variable $Y(\varphi,\spa)$ which is determined by its moments (it has a convergent moment-generating function) such that we have the convergence in law
 	$$\frac{\Phi(\spa_n)-\Phi(\spa)}{n}\rightharpoonup_{n \to \infty} Y(\varphi,\spa);$$
 	see Theorem \ref{theo:singular_case}.
 \end{itemize} 
}\bigskip

\section{The method of cumulants}\label{sec:method_cumulants}
In this section, we recall the notion of (joint) cumulants of random variables and the results from \cite{feray2016mod,feray2017mod}, which relate the existence of a sparse dependency graph for a family of random variables to the size of the cumulants and to the fluctuations of their sum.

\subsection{Joint cumulants}
A \emph{set partition} of $[\![1,n ]\!]$ is a family of non-empty disjoint subsets of $[\![1,n ]\!]$ (the \emph{parts} of the partition), whose union is $[\![1,n ]\!]$. For instance,
\begin{align*}
\{\{1,4,8\},\{3,5,6\},\{2,7\},\{9\}\}
\end{align*}
is a set partition of $[\![1,9]\!]$. We denote $\mathfrak{Q}(n)$ the set of set partitions of $[\![1,n ]\!]$. It is endowed with the \textit{refinement} order: a set partition $\pi$ is \textit{finer} than another set partition $\pi'$ if every part of $\pi$ is included in a part of $\pi'$. Denote $\mu$ the Möbius function of the partially ordered set $(\mathfrak{Q}(n),\preceq)$ (see \cite{Rot64}). One has
\begin{align*}
\mu(\pi) \vcentcolon = \mu(\pi, \{[\![1,n]\!]\}) = (-1)^{\ell(\pi)-1}\,(\ell(\pi)-1)!,
\end{align*}
where $\ell(\pi)$ is the number of parts of $\pi$; see \cite[Chapter 3, Equation (30) p.~128]{Stan97}.\medskip

Given a probability space $(\Omega,\mathcal{F}, \mathbb{P})$ , we set 
$$\mathscr{A} = \bigcap\limits_{p \in \mathbb{N}^*} \mathscr{L}^p(\Omega, \mathcal{F}, \mathbb{P}),$$ 
which has a structure of real algebra. For any integer $r \geq 1$, we define a  map $\kappa_r \colon   \mathscr{A}^r \to \R $ by
\begin{align*}
 \kappa_r(X_1,\dots,X_r) = \left[ t_1 \cdots t_r\right] \, \mathrm{log} \left( \esper\left[\E^{t_1 X_1 + \cdots + t_r X_r} \right] \right) \quad \text{for} \enskip (X_i)_{i \in [\![1,r]\!]} \in \mathscr{A}^r,
\end{align*}
where $\left[t_1 \cdots t_r\right](F)$ is the coefficient of the monomial $\prod_{i=1}^r t_i$ in the series expansion of $F$. Here, $\mathrm{log} (\esper[\E^{t_1 X_1 + \cdots + t_r X_r} ] )$ is considered as a formal power series whose coefficients are polynomials in the joint moments of the $X_i$'s; we do not ask \emph{a priori} for the convergence of the exponential generating function. We call the map $\kappa_r$ the $r$-th joint cumulant map, and we define the \emph{joint cumulant map} 
$$\kappa \colon \bigcup\limits_{r \in \mathbb{N}^*} \mathscr{A}^r \to \mathbb{R}$$
 by $ {\kappa|}_{\mathscr{A}^r} = \kappa_r$ for any integer $r \geq 1$. For a specific sequence $(X_i)_{i \in [\![1,r]\!]} \in \mathscr{A}^r$, we call the quantity $\kappa_r((X_i)_{i \in [\![1,r]\!]} )$ the joint cumulant of $(X_i)_{i \in [\![1,r]\!]} \in \mathscr{A}^r$. This notion of joint cumulant was introduced by Leonov and Shiryaev in \cite{leonov1959method}, and it generalises the usual cumulants: for $X \in \mathscr{A}$,
$$\kappa^{(r)}(X) \vcentcolon= \kappa_r(X,\dots,X)$$ 
is the usual $r$-th cumulant of $X$, that is $r!\,[t^r](\log \esper[\E^{tX}])$. We summarise the properties of the map $\kappa$ in the following:

\begin{prop}\label{prop:cumulants}
\begin{enumerate}
	\item The map $\kappa$ is multilinear.
	\item The joint cumulants and the joint moments are related by the poset of set partitions, and the following formulas hold:
\begin{align*}
&\esper\left[X_1 \cdots X_r\right] = \sum \limits_{\pi \in \mathfrak{Q}(r)} \prod \limits_{C \in \pi} \kappa\left( X_i \, ; \, i \in C\right); \\
&\kappa(X_1,\dots, X_r) = \sum \limits_{\pi \in \mathfrak{Q}(r)} \mu(\pi) \prod\limits_{C \in \pi} \esper\left[\prod\limits_{i \in C} X_i\right].
\end{align*}
\item If the variables $X_1, \dots,X_r$ can be split into two non-empty sets of variables which are independent of each other, then $\kappa(X_1,\dots,X_r)$ vanishes.
\end{enumerate}
\end{prop}

For example, the joint cumulants of one or two variables are respectively the expectation and the covariance:
\begin{align*}
\kappa(X_1) = \esper[X_1] \, ; \,\, \kappa(X_1,X_2) = \esper[X_1 X_2] - \esper[X_1]\esper[X_2].
\end{align*}
\PLM{For the convenience of the reader, we also recall the value of the third cumulant: $\kappa(X_1,X_2,X_3) = \esper[X_1X_2X_3] - \esper[X_1]\esper[X_2X_3] - \esper[X_2]\esper[X_1X_3]- \esper[X_3]\esper[X_1X_2] +2\esper[X_1]\esper[X_2]\esper[X_3]$.}
\medskip

\subsection{Dependency graphs and bounds on cumulants}
\PLM{A real random variable $X$ is distributed according to the normal law $\mathcal{N}(m,\sigma^2)$ with mean $m$ and variance $\sigma^2$ if and only if $\kappa^{(1)}(X)=m$, $\kappa^{(2)}(X)=\sigma^2$ and $\kappa^{(r)}(X)=0$ for $r \geq 3$. More generally, a sequence of random variables $(X_n)_{n \in \N}$ converges in distribution towards a normal law $\mathcal{N}(m,\sigma^2)$ if the two first cumulants $\kappa^{(1,2)}(X_n)$ converge toward $m$ and $\sigma^2$ respectively, and if $\lim_{n \to \infty} \kappa^{(r)}(X_n)=0$ for $r \geq 3$; see for instance \cite[Theorem 1]{Jan88}. In the series of papers \cite{feray2016mod,feray2017mod,2017arXiv171206841F,BMN19}, a method of cumulants has been built in order to make more precise this result of asymptotic normality, assuming that one has good upper bounds on the size of the cumulants of the random variables $X_n$. This method falls in the framework of \emph{mod-Gaussian convergence} also constructed in the aforementioned papers. We recall below the main results from this theory; see \cite[Definition 2 and Theorem 3]{2017arXiv171206841F}.
}

\begin{defn}
Let $(S_n)_{n \in \mathbb{N}}$ be a sequence of real-valued random variables. We fix $A \geq 0$, and we consider two positive sequences $(D_n)_{n \in \mathbb{N}}$ and $(N_n)_{n \in \mathbb{N}}$  such that $$\lim_{n \to \infty} \frac{D_n}{N_n} = 0\quad(\text{hypothesis of sparcity}).$$
The hypotheses of the method of cumulants with parameters $((D_n)_{n \in \mathbb{N}}, (N_n)_{n \in \mathbb{N}}, A)$ and with limits $(\sigma^2,L)$ for the sequence $(S_n)_{n \in \mathbb{N}}$ are the two following conditions:
\begin{itemize}
\item For any $r \geq 1$, we have:
\begin{align*}
|\kappa^{(r)}(S_n)| \leq N_n (2D_n)^{r-1} r^{r-2} A^r.
\end{align*}
\item There exist two real numbers $\sigma^2 \geq 0$ and $L$ such that:
\begin{align*}
&\frac{\kappa^{(2)}(S_n)}{N_n D_n} = (\sigma_n)^2 = \sigma^2 \left( 1 + o\left(\left(\frac{D_n}{N_n}\right)^{\!1/3}\right)\right); \\
&\frac{\kappa^{(3)}(S_n)}{N_n (D_n)^2} = L_n = L(1+o(1)) .
\end{align*}
\end{itemize}
\end{defn}
\noindent In particular, the first estimate in the second item states that the variance of $S_n$ is equivalent to $\sigma^2\,N_nD_n$.

\begin{theo}\label{theo:esti_cumu}
Let $(S_n)_{n \in \mathbb{N}}$ be a sequence of real-valued random variables that satisfies the hypotheses of the method of cumulants, with parameters $((D_n)_{n \in \mathbb{N}}, (N_n)_{n \in \mathbb{N}}, A)$ and with limits $(\sigma^2,L)$. Assuming that $\sigma^2 >0$, we set:
\begin{align*}
 Y_n = \frac{S_n - \esper[S_n]}{\sqrt{\var(S_n)}}.
\end{align*}
\begin{enumerate}
\item Central limit theorem with an extended zone of normality: we have $Y_n {\rightharpoonup}_{n \to \infty}\, \mathcal{N}_{\mathbb{R}}(0,1)$, and more precisely,
\begin{align*}
\proba\left[Y_n \geq y_n\right] = \proba\left[\mathcal{N}_{\mathbb{R}}(0,1) \geq y_n\right](1+o(1))
\end{align*}
for any sequence $(y_n)_{n \in \mathbb{N}}$ with $|y_n|  \ll \left(\frac{N_n}{D_n}\right)^{1/6}$.
\item Berry--Esseen type bound: the Kolmogorov distance between $Y_n$ and the standard Gaussian distribution satisfies 
\begin{align*}
d_{\mathrm{Kol}}(Y_n,\mathcal{N}(0,1)) \leq \frac{C\, A^3}{(\sigma_n)^3} \sqrt{\frac{D_n}{N_n}},
\end{align*}
where $C=76.36$ is a universal constant.
\item Moderate deviations: for any sequence $(y_n)_{n \in \mathbb{N}}$ with $1\ll y_n \ll \left(\frac{N_n}{D_n}\right)^{1/4}$,
\begin{align*}
\proba\left[Y_n \geq y_n\right]  = \frac{\E^{-\frac{(y_n)^2}{2}}}{y_n \sqrt{2 \pi}} \exp\left(\frac{L}{6 \sigma^3} \sqrt{\frac{D_n}{N_n}}(y_n)^3\right)(1+o(1)).
\end{align*}
\item Local limit theorem: for any $y \in \mathbb{R}$, any Jordan measurable set $B$ with positive Lebesgue measure $\mathrm{m}(B) > 0$, and any real exponent $\delta$ in $(0,\frac{1}{2})$,
\begin{align*}
\lim\limits_{n \to \infty} \left(\frac{N_n}{D_n}\right)^{\!\delta} \,\proba\!\left[Y_n-y \in \left(\frac{D_n}{N_n}\right)^{\!\delta} B \right] = \frac{1}{\sqrt{2\pi}}\,\E^{-\frac{y^2}{2}} \,\mathrm{m}(B).
\end{align*}
\PLM{\item Concentration inequality: suppose that in addition to the hypotheses of the method of cumulants, we have almost surely $|S_n| \leq N_nA$. Then, for any $x\geq 0$ and any $n \in \N$,
$$\proba[|Y_n|\geq x]\leq 2\exp\left(-\frac{(\sigma_n)^2x^2}{9A^2}\right). $$}
\end{enumerate}
\end{theo}
\noindent This list of results corresponds to Theorem 9.5.1 in \cite{feray2016mod} (CLT and moderate deviations), Corollary 30 in \cite{feray2017mod} (Kolmogorov distance), Proposition 4.9 in \cite{BMN19} (local limit theorem), and Proposition 6 in \cite{2017arXiv171206841F} (concentration inequality). \bigskip

We shall use the method of dependency graphs in order to verify the hypothesis of the previous theorem.
Let $S = \sum_{v \in V} A_v$ be a finite sum or real-valued random variables. We say that a graph $G=(V,E)$ is a \emph{dependency graph} for the family of random variables $(A_v)_{v \in V}$ if, given two disjoint subsets $V_1,V_2 \subseteq V$, if there is no edge $e = (v,w) \in E$ such that $v \in V_1$ and $w \in V_2$, then the two vectors $(A_v)_{v \in V_1}$ and $(A_w)_{w \in V_2}$ are independent.

\begin{theo}\label{theo:bounds_dependency}
Let $S = \sum_{v \in V} A_v$ be a sum of random variables such that $(A_v)_{v \in V}$ admits a dependency graph $G=(V,E)$, with $$N = \mathrm{card}(V)\qquad;\qquad D = 1 + \max_{v \in V}(\mathrm{deg}(v)).$$
We also assume that $|A_v| \leq A$ almost surely for any $v$ in $V$. Then, for any $r \geq 1$,
\begin{align}\label{bound:cumulant}
|\kappa^{(r)}(S)| \leq N (2D)^{r-1} r^{r-2} A^r.
\end{align}
\end{theo}
\PLM{We refer to \cite[Theorem 9.1.7]{feray2016mod} for a proof of this result; later, we shall recall some of its arguments and adapt them in order to obtain adequate bounds on the cumulants of polynomials of the Gromov--Prohorov sample model of a compact homogeneous space.}\bigskip

\section{Generic fluctuations of the sample model}\label{sec:generic_fluctuations}
Throughout this section, $\spa = \spaex \in \mathbb{M}$ is a fixed metric measure space and $\Phi^{p,\varphi} \in \Pi$ a fixed polynomial. As in Section \ref{sub:almost_sure_convergence}, we denote $\spa_n$ the sample model of $\spa$ with $n$ independent points $X_1,\ldots,X_n$, and we are going to study the convergence of $\Phi (\spa_n)$ toward $\Phi(\spa)$. 

\subsection{Dependency graphs for the sample model}
\noindent For any sequence $X \colon \mathbb{N} \to E$ with values in a set $E$ and for any map $f \colon S \to \mathbb{N}$, we denote by $X_f$ the map $X \circ f$.
For example, if we take $f = I \in [\![1,n]\!]^5$ which is a 5-tuple, we have $X_I = (X_{I_1},X_{I_2},X_{I_3},X_{I_4},X_{I_5})$.
For any finite or infinite sequence $I \colon S \to T$, we write
\begin{align*}
d(X_I) = (d(X_{I_i},X_{I_j}))_{i \in S, j \in S}
\end{align*}
We see a $p$-tuple $\bar{\imath}$ as a map $\bar{\imath} \colon [\![1,p ]\!] \to [\![1,n ]\!]$ and we denote by $\overline{\mathrm{Im}}(\ibar)$ the multiset-image of this map, taking as a multiplicity function the map $m \colon \mathrm{Im}(\ibar) \to \mathbb{N}$ defined for any $\overline{\mathrm{Im}}(\ibar)$ by $m(x) = \mathrm{Card}((\ibar)^{-1}(x))$.
We have 
\begin{align*}
\forall n \geq 1, \, \Phi(\spa_n) = \frac{1}{n^p} \,\sum_{\ibar \in [\![1,n ]\!]^p } \varphi(d(X_{\ibar})).
\end{align*}

\noindent We write  $S_{n}(\varphi,\spa) = n^p\,\Phi(\spa_n) = \sum_{\ibar \in [\![1,n ]\!]^p } \varphi(d(X_{\ibar}))$, which is a sum of dependent random variables. We are going to use the method of cumulants in order to study the asymptotic probabilistic behavior of $S_{n}(\varphi,\spa)$. Placing ourselves in the framework of the previous section, we take $V = [\![1,n ]\!]^p$, $S = S_{n}(\varphi,\spa) = \sum_{\ibar \in [\![1,n ]\!]^p } \varphi(d(X_{\ibar})) $, $A = \|\varphi\|_{\infty}$, and two vertices $\ibar$ and $\jbar$ will be adjacent in the graph $G = (V,E)$ if and only if they have at least one index in common, \emph{i.e.} if and only if 
$$\mathrm{Card}\left(\overline{\mathrm{Im}}(\ibar)\cap \overline{\mathrm{Im}}(\jbar)\right) \geq 1 .$$
\PLM{\begin{lemma}
The condition written above defines a dependency graph for the family of random variables $(\varphi(d(X_{\ibar})))_{\ibar \in V}$.
\end{lemma}
\begin{proof}
Suppose that $\{\ibar^{1},\ldots,\ibar^{r}\}$ and $\{\jbar^{1},\ldots,\jbar^{s}\}$ are two sets of $p$-tuples which are not connected. Then, there is no index $i$ belonging to an intersection $\overline{\mathrm{Im}}(\ibar^{a})\cap \overline{\mathrm{Im}}(\jbar^{b})$, so the two sets of variables
$$\left\{X_i,\,\,i \in \bigcup_{a=1}^r\overline{\mathrm{Im}}(\ibar^{a}) \right\} \quad\text{and}\quad \left\{X_j,\,\,j \in \bigcup_{b=1}^s\overline{\mathrm{Im}}(\jbar^{b}) \right\}$$
are disjoint. As the two vectors $(\phi(d(X_{\ibar^{a}})))_{1\leq a \leq r}$ and $(\phi(d(X_{\jbar^{b}})))_{1\leq b \leq s}$ are measurable functions of these two sets, they are independent.
\end{proof}}

In the dependency graph $G$ constructed above, we have $N = n^p$ and $D \leq p^2 n^{p-1}$. Indeed, we can build a surjective map from $[\![1,p ]\!]^2 \times[\![1,n ]\!]^{p-1}$  to the set of adjacent vertices of a vertex $\bar{\imath} \in V $ taking 
\begin{align*}
[\![1,p ]\!]^2 \times [\![1,n ]\!]^{p-1} &\to \{\text{adjacent vertices of }\ibar \} \\
(i,j,(y_k)_{k \neq j}) &\mapsto (y'_k)_{k \in [\![1,p ]\!]}
\end{align*}
with $y'_k = y_k$ if $k\neq j$ and $y'_j = x_i$. 
Therefore, we have from Theorem~\ref{theo:bounds_dependency}:
\begin{align*}
\forall r \geq 1, |\kappa^{(r)}(S_n(\varphi,\spa))| \leq n^p (2p^2 n^{p-1})^{r-1} r^{r-2} (\|\varphi\|_\infty)^r.
\end{align*}
which is an upper bound of order $n^{(p-1)r + 1}$.\medskip


\subsection{Polynomiality of the cumulants}
For any $r \geq 1$, we can write by multilinearity of the cumulant:
\begin{align*}
\kappa^{(r)}(S_n(\varphi,\spa)) = \sum_{(\ibar^1,\dots,\ibar^r) \in V^r} \kappa\left(\varphi(d(X_{\ibar^1})),\dots,\varphi(d(X_{\ibar^r}))\right).
\end{align*}
For any $I = (\ibar^1,\dots,\ibar^r) \in V^r$, we set $\varphi(d(X_I)) = (\varphi(d(X_{\ibar^1})),\dots,\varphi(d(X_{\ibar^r})))$, hence:
\begin{align*}
\kappa^{(r)}(S_n(\varphi,\spa)) = \sum_{I \in V^r} \kappa (\varphi(d(X_I))).
\end{align*}
We identify here $V^r = \left([\![1,n ]\!]^p \right)^r $ with the set $[\![1,n ]\!]^{pr}$ by preserving the \textit{lexicographic} order: \emph{i.e.} by using the bijection
\begin{align*}
b \colon &[\![1,r ]\!] \times \![1,n ]\!] \to [\![1,rp]\!]
\\
&(k,l) \mapsto (k-1)p + l.
\end{align*}
\begin{prop}\label{prop:polynomiality}
For any integer $r \geq 1$, the map
\begin{align*}
\mathbb{N}^* &\to \mathbb{R} \\
n &\mapsto \kappa^{(r)}(S_n(\varphi,\spa))
\end{align*}
is a polynomial in $\mathbb{R}[n]$ with degree not exceeding $(p-1)r+1$.
\end{prop}

\begin{proof}
 For $x= (x_1,\dots,x_{pr})$ in $[\![1,n ]\!]^{pr}$, we consider the equivalence relation $\pi_x$ over $[\![1,pr]\!]$ defined by $i \sim j$ if and only if $x_i = x_j$. We then denote $\mathrm{Sp}_n(x)$ the set-partition in $\mathfrak{Q}(pr)$ associated to the equivalence relation $\pi_x$. 
Given two families of indices $I=(x_1,\ldots,x_{pr})$ and $J=(y_1,\ldots,y_{pr})$ in $[\![1,n ]\!]^{pr}$, note that if $\mathrm{Sp}_n(I) = \mathrm{Sp}_n(J)$, then $\kappa (\varphi(d(X_I))) =\kappa (\varphi(d(X_J)))$. Indeed, if $\mathrm{Sp}_n(I) = \mathrm{Sp}_n(J)$, then one can find a bijection $\psi : [\![1,n]\!] \to [\![1,n]\!]$ such that $\psi(x_a)=y_a$ for any $a \in [\![1,pr]\!]$; the result follows since the $X_{i}$'s all have the same law. Given $\pi \in \mathfrak{Q}(pr)$, we denote:
\begin{equation}
	\kappa(\pi,\varphi) = \kappa (\varphi(d(X_I))) 
	\quad \text{for any $I \subset [\![1,n ]\!]^{pr}$ such that $\mathrm{Sp}_n(I) = \pi$}. \label{eq:kappa_pi}
\end{equation} 
Then,
\begin{align*}
\kappa^{(r)}(S_n(\varphi,\spa)) &= \sum_{I \in V^r} \,\kappa (\varphi(d(X_I)))
\\
&= \sum_{\pi \in \mathfrak{Q}(pr)} \mathrm{Card}(\pi,n) \,\kappa(\pi,\varphi).
\end{align*}
where $\mathrm{Card}(\pi,n)$ denotes the number of families $I \in [\![1,n]\!]^{pr}$ such that $\mathrm{Sp}_n(I) = \pi$. We now remark that given 
$\pi \in \mathfrak{Q}(pr) $,
$\mathrm{Sp}_{n}^{-1}(\pi)$ is in bijection with the set 
\begin{align*}
\{(x_1,\dots,x_{\ell(\pi)}) \, ; \,  x_i \in [\![1,n]\!] \text{ and for all } i \neq j \in [\![1,{\ell(\pi)} ]\!],  x_i \neq x_j \}
. 
\end{align*}
The cardinal of this set is $n^{\downarrow {\ell(\pi)}} = n(n-1)\cdots (n-({\ell(\pi)}-1))$ (this is valid even if $n < \ell(\pi)$). Thus, for any $n \geq 1$,
\begin{equation}
\kappa^{(r)}(S_n(\varphi,\spa)) = \sum_{\pi \in \mathfrak{Q}(pr)}\kappa(\pi,\varphi) \,n^{\downarrow \ell(\pi)}. \label{eq:expansion_cumulant}
\end{equation}
This proves the polynomiality, and since we know that the left-hand side is a $O(n^{(p-1)r+1})$, the degree of the polynomial is smaller than $(p-1)r+1$.
\end{proof}

In Equation \eqref{eq:expansion_cumulant}, we know that the terms with degree strictly larger than $(p-1)r+1$ cancel one another. Let us give a simpler explanation of this vanishing:
\begin{prop}\label{prop:vanish generic}
If $r\geq 2$ and $\ell(\pi) > (p-1)r + 1$, then $\kappa(\pi,\varphi) = 0$.
\end{prop}

\begin{proof}
This is mostly a rewriting of the proof of the general upper bound on cumulants stated in Theorem \ref{theo:esti_cumu}. For the convenience of the reader, let us give a proof which is adapted to our situation; this will also enable us to introduce combinatorial objects which will play a major role in Section \ref{sec:homogeneous_case}. Given $\pi\in \mathfrak{Q}(pr)$, we construct a graph $G_\pi$ on the vertex set $V({G_\pi}) = [\![1,pr]\!]$ as follows. For any part $A$ of the set partition $\pi$, we associate a spanning tree $T_A$ of the set of vertices $A$, then we define $G_\pi$ as the disjoint union of those spanning trees. We have $\sum_{A \in \pi} (|E(T_A)|+1) = pr$. This implies $|E(G_\pi)|=\sum_{A \in \pi} |E(T_A)| \leq r-2$ by the assumption on $\ell(\pi)$. We now construct a multigraph $H_\pi$ with vertex set $V(H_\pi) = [\![1,r]\!]$, by contracting the vertices of the graph $G_\pi$ according to the map
\begin{align*}
(b^{-1})_1 : [\![1,rp]\!] &\to [\![1,r]\!]\\
(k-1)p+l &\mapsto k.
\end{align*}
The multigraph $H_\pi$ has the same number of edges as $G_\pi$, so $E(H_\pi)=E(G_\pi) \leq r-2$ and $H_\pi$ is not connected. As a consequence, if $[\![1,r]\!]=A \sqcup B$ are two non-connected components and $I=(\ibar^1,\ldots,\ibar^r)$ is a family of indices such that $\mathrm{Sp}_n(I)=\pi$, then the two families of indices $\bigcup_{a \in A} \ibar^a$ and $\bigcup_{b \in B} \ibar^b$ are disjoint. This implies that $\kappa(\pi,\varphi) = 0$, by using the third property in Proposition \ref{prop:cumulants}. 
\end{proof}\medskip

\subsection{Limiting variance and asymptotics of the fluctuations}

In order to apply Theorem \ref{theo:esti_cumu}, we also have to compute the limiting parameters $\sigma^2$ and $L$ involved in the method of cumulants. Identifying the leading terms in Equation \eqref{eq:expansion_cumulant}, we obtain:
\begin{align*}
\frac{\kappa^{(2)}(S_n(\varphi,\spa))}{N_n D_n} = \frac{\kappa^{(2)}(S_n(\varphi,\spa))}{p^2 n^{2p-1}} = \frac{1}{p^2} \sum_{\substack{\pi \in \mathfrak{Q}(2p) \\ \ell(\pi) = 2p-1}} \kappa(\pi,\varphi)  + O\left(\frac{1}{n}\right).
\end{align*}
For $k,l \in [\![1,p]\!]$, we define the partition 
\begin{equation}
\pi_{k,l} = \{k,l+p\} \cup \{ \{t\} \, ; \, t \in [\![1,2p]\!] \setminus \{k,l+p\} \}=\,  \begin{tikzpicture}[scale=0.5, baseline=1mm]
\foreach \x in {0,1,2,3,4,5}
{\fill (\x,0) circle (3pt);
\fill (\x,1) circle (3pt);}
\draw (1,1) -- (3,0);
\draw (1,1.5) node {\footnotesize $k$};
\draw (3,-0.5) node {\footnotesize $l$};
\end{tikzpicture}
\,\,\,;\label{eq:pi_kl}
\end{equation}
the picture above of the set partition makes appear the integers in $[\![1,p]\!]$ on the top row, and the integers in $[\![p+1,2p]\!]$ on the bottom row. We then have:
\begin{align*}
\frac{\kappa^{(2)}(S_n(\varphi,\spa))}{N_n D_n} = \frac{1}{p^2} \sum_{1 \leq k,l \leq p} \kappa(\pi_{k,l},\varphi)  + O\left(\frac{1}{n}\right).
\end{align*}
Indeed, a set partition $\pi$ of $[\![1,2p]\!]$ with length $2p-1$ consists of a pair $\{k,l\}$ and of singletons, and if the pair $\{k,l\}$ is included in $[\![1,p]\!]$ or in $[\![p+1,2p]\!]$, then the graph $H_{\pi}$ introduced during the proof of Proposition \ref{prop:vanish generic} is not connected (it is the graph on $2$ vertices and without edge), so $\kappa(\pi,\varphi)=0$.
Similarly, we compute the limiting third cumulant $L$:
\begin{align*}
\frac{\kappa^{(3)}(S_n(\varphi,\spa))}{N_n (D_n)^2} =\frac{\kappa^{(3)}(S_n(\varphi,\spa))}{p^4 n^{3p-2}} =\frac{1}{p^4} \sum_{\substack{\pi \in \mathfrak{Q}(3p) \\ \ell(\pi) = 3p-2}} \kappa(\pi,\varphi)  + O\left(\frac{1}{n}\right).
\end{align*}
For $i,j,k,l \in [\![1,p]\!]$ with $j \neq k$, we define the partition:
\begin{align}
\pi_{i,j,k,l} &= \{i,j+p\} \cup \{k+p,l+2p\} \cup \{ \{t\} \, ; \, t \in [\![1,3p]\!] \setminus \{i,j+p,k+p,l+2p\} \} \nonumber\\
&=\,\, \begin{tikzpicture}[scale=0.5, baseline=4mm]
\foreach \x in {0,1,2,3,4,5}
{\fill (\x,0) circle (3pt);
\fill (\x,1) circle (3pt);
\fill (\x,2) circle (3pt);}
\draw (1,1) -- (3,0);
\draw (4,1) -- (5,2);
\draw (0.7,1.2) node {\footnotesize $k$};
\draw (4.3,0.8) node {\footnotesize $j$};
\draw (5,2.5) node {\footnotesize $i$};
\draw (3,-0.5) node {\footnotesize $l$};
\end{tikzpicture}\,\,,\label{eq:pi_ijkl}
\end{align}
and if $j=k$:
\begin{align*}
\pi_{i,j,j,l} = \{i,j+p,l+2p\} \cup \{ \{t\} \, ; \, t \in [\![1,3p]\!] \setminus \{i,j+p,l+2p\} \} = \,\, \begin{tikzpicture}[scale=0.5, baseline=4mm]
\foreach \x in {0,1,2,3,4,5}
{\fill (\x,0) circle (3pt);
\fill (\x,1) circle (3pt);
\fill (\x,2) circle (3pt);}
\draw (2,1) -- (3,0);
\draw (2,1) -- (5,2);
\draw (1.5,1) node {\footnotesize $j$};
\draw (5,2.5) node {\footnotesize $i$};
\draw (3,-0.5) node {\footnotesize $l$};
\end{tikzpicture}\,\,.
\end{align*}
\PLM{These are the only possible forms for a set partition of $[\![1,3p]\!]$ with length $3p-2$ and with the condition that $H_\pi$ is connected. For the $\pi_{i,j,k,l}$'s with $j \neq k$, we also need to take into account the set partitions where two elements of the top row or of the bottom row (instead of the middle row) are connected to elements of the other rows; this leads to a factor $3$ in the enumeration}.
Thus, we have:
\begin{align*}
\frac{\kappa^{(3)}(S_n(\varphi,\spa))}{N_n (D_n)^2} = \frac{1}{p^4} \sum_{1 \leq i,j,k,l \leq p} c_{i,j,k,l}\,\kappa(\pi_{i,j,k,l},\varphi)  + O\!\left(\frac{1}{n}\right)
\end{align*}
with 
\begin{equation}
c_{i,j,k,l} = \begin{cases}
	3 &\text{if }j\neq k,\\
	1 & \text{if }j=k.
\end{cases}\label{eq:coefficients_cijkl}
\end{equation}
\PLM{Similar formulas were obtained in \cite[Section 5]{2017arXiv171206841F} for the limiting behavior of the first cumulants of observables of random graphs associated to a graphon parameter. We have now established:}

\begin{theo}[Fluctuations in the generic case]\label{theo:generic_case}
Let $\spa = \spaex \in \mathbb{M}$ a metric measure space and $\Phi=\Phi^{p,\varphi} \in \Pi$ a polynomial.
\begin{enumerate}
\item The random variable $S_n(\varphi,\spa) = n^p\,\Phi(\spa_n)$ satisfies the hypotheses of the method of the cumulants 
\begin{itemize}
	\item with parameters $D_n = p^2 n^{p-1}$, $N_n = n^p$ and $A = \|\varphi\|_{\infty}$,
	\item and with limits $\sigma^2 = \frac{1}{p^2} \sum_{1 \leq k,l \leq p} \kappa(\pi_{k,l},\varphi)$ and $L = \frac{1}{p^4} \sum_{1 \leq i,j,k,l \leq p} c_{i,j,k,l}\,\kappa(\pi_{i,j,k,l},\varphi)$.
\end{itemize} 
In the formul{\ae} for $\sigma^2$ and $L$, $\kappa(\pi,\varphi)$ with $\pi $ set partition of $[\![1,pr]\!]$ is defined by Equation \eqref{eq:kappa_pi}, and the coefficients $c_{i,j,k,l}$ are given by Equation \eqref{eq:coefficients_cijkl}; the diagrams of the set partitions are drawn in Equations \eqref{eq:pi_kl} and \eqref{eq:pi_ijkl}.
\item  If $\sigma(\varphi,\spa) > 0$, then the random variables
\begin{align*}
Y_n(\varphi,\spa) = \frac{\Phi(\spa_n) - \Phi(\spa)}{\sqrt{\var(\Phi(\spa_n))}}
\end{align*}
satisfy all the limiting results from Theorem~\ref{theo:esti_cumu}. In particular, we have the convergence in law $Y_n(\varphi,\spa) \rightharpoonup_{n \to \infty} \mathcal{N}(0,1)$, and
$$d_{\mathrm{Kol}}\left(Y_n(\varphi,\spa),\mathcal{N}(0,1)\right) = O\left(\left(\frac{\|\varphi\|_\infty}{\sigma}\right)^3\,p\,n^{-1/2}\right).$$
Under the assumption $\sigma(\varphi,\spa)>0$, the renormalisation $\sqrt{\var(\Phi(\spa_n))}$ is of order $n^{-1/2}$, and more precisely, $\var(\Phi(\spa_n))$ is a polynomial in $n^{-1}$ without constant term, and with leading term $p^2\,\sigma^2(\varphi,\spa)\,n^{-1}$.
\end{enumerate}
\end{theo}

\PLM{With the terminology of \cite[Section 6, Definition 30]{2017arXiv171206841F}, the theorem above ensures that the pair $(\mathbb{M},\Pi)$ is a \emph{mod-Gaussian moduli space}: generically (as soon as $\sigma(\varphi,\spa)>0$), an observable of the Gromov--Prohorov sample model of a mm-space $\spa$ has normal fluctuations of size $O(n^{-1/2})$, and the limiting variance $\sigma^2(\varphi,\spa)$ writes as an observable $\kappa_2(\varphi,\varphi) \in \Pi$ evaluated on the mm-space $\spa$. In this setting, a general problem is to identify the singular points of the space $\mathbb{M}$, that is to say the mm-spaces such that $\sigma^2(\varphi,\spa)=0$ for any function $\varphi \in \Ccal_b(\R^{\binom{p}{2}})$, and thus such that the fluctuations of $\Phi^{p,\varphi}(\spa_n)$ are of order smaller than $n^{-1/2}$. The next sections of this paper are devoted to this topic.}
\bigskip

\section{Fluctuations in the homogeneous case}\label{sec:homogeneous_case}

In this section, we place ourselves in the singular case of the Gromov--Prohorov sample model, where 
\begin{align}
\forall p \geq 1,\,\,\forall \varphi \in \Ccal_b\left(\mathbb{R}^{\binom{p}{2}}\right),\,\,\, \sigma^2 (\varphi,\spa) = \frac{1}{p^2} \sum_{1 \leq k,l \leq p} \kappa(\pi_{k,l},\varphi) = 0.\label{eq:singular_case}
\end{align}
This implies that $\frac{\Phi(\spa_n)-\Phi(\spa)}{\sqrt{n}}$ converges in probability to $0$ for any observable $\Phi \in \Pi$. A condition which implies \eqref{eq:singular_case} and which is much easier to check is:
\begin{align}\label{eq:cov_vanish}
\forall p \geq 1,\,\,\forall k,l \in [\![1,p]\!],\,\,\forall \varphi \in \Ccal_b\left(\mathbb{R}^{\binom{p}{2}}\right), \,\,\,\,\kappa(\pi_{k,l},\varphi) = 0.
\end{align}
It is not known whether it is possible to have \eqref{eq:singular_case} without having \eqref{eq:cov_vanish}. We strongly believe that these two conditions are actually equivalent; let us detail a bit why this should be true. In Section \ref{sec:circle}, we shall introduce monomial observables of mm-spaces which are indexed by finite multigraphs; Equations \eqref{eq:singular_case} and \eqref{eq:cov_vanish} correspond to relations between the values of these observables on a mm-space. This viewpoint leads then to questions of graph theory, and a combinatorial study of these relations should allow one to understand whether Condition \eqref{eq:cov_vanish} is strictly stronger than, or equivalent to Condition \eqref{eq:singular_case}; we aim to address this problem in a forthcoming paper. Let us mention that a analogous problem occurs in the study of fluctuations of graphon models, where the Erd\H{o}s--Rényi random graphs are singular models but may not be the only singular points; see again \cite{2017arXiv171206841F}. In the remainder of the article, we assume that Condition \eqref{eq:cov_vanish} is satisfied, and we prove the following results:
\begin{enumerate}
	\item This probabilistic condition is equivalent to a geometric property for the space $\spa$, namely, $\spa$ is a compact homogeneous space $G/K$ on which the compact group $G$ acts by isometry; see Theorem \ref{theo:homogeneous}.
	\item In this situation, for any observable $\Phi \in \Pi$, $n(\Phi(\spa_n)-\Phi(\spa))$ converges in distribution toward a law which is determined by its moments (Theorem \ref{theo:singular_case}).
	\item The limiting distribution is not necessarily Gaussian; we provide in Section \ref{sec:circle} an explicit example when $\spa$ is the circle.
\end{enumerate}

Let us introduce a few more notations. Given $\spa=\spaex \in \mathbb{M}$, we denote $\mathrm{Isomp}(\spa)$ the group of isometries $i : \mathcal{X} \to \mathcal{X}$ which are measure-preserving: 
  $$d(\cdot,\cdot) = d(i(\cdot),i(\cdot))\quad \text{and}\quad i_*\mu=\mu.$$
The group $\mathrm{Isomp}(\spa)$ is endowed
with the topology of uniform convergence on compact subsets, which is defined by the neighborhoods
\begin{align}
V(i,K,\epsilon) = \left\{ j \in \mathrm{Isomp}(\spa) \, , \, \sup\limits_{x\in K} d(i(x),j(x)) < \epsilon \right\}\label{eq:neighborhood}
\end{align} 
for $i \in \mathrm{Isomp}(\spa)$, $K$ compact subset of $\mathcal{X}$ and $\epsilon > 0$. The group action of $G=\mathrm{Isomp}(\spa)$ on $\mathcal{X}$ is the continuous map
\begin{align*}
\mathrm{Isomp}(\spa) \times \mathcal{X} &\to \mathcal{X} \\
(g,x) &\mapsto g \cdot x = g(x).
\end{align*}
The orbit of $x \in \mathcal{X}$ is 
$O_x = \left\{y\in \mathcal{X}\mid\exists g \in G : y = g \cdot x \right\}$, and the stabilizer of $x$ is the subgroup of $G$ given by $\mathrm{St}_x = \{ g \in \mathrm{G} \mid g \cdot x = x\}
$. For a subgroup $K$ of a group $G$, we denote by $G/K$ the space of left cosets of the group $G$ over $K$, and 
\begin{align*}
\pi : G &\to G/K\\
g &\mapsto \bar{g}= gK
\end{align*}
 the canonical projection map. The group action by left translations of $G$ on $G/K$ is $g \cdot \bar{g_1} = \overline{gg_1}$. For any $x \in \spa$, we have the bijection
\begin{align*}
\left\{\begin{array}{ccc} G/\mathrm{St}_x & \to & O_x\\ \bar{g} & \mapsto & g \cdot x. \end{array}\right.
\end{align*}
Finally, we denote $\mathcal{X}_\mu^\mathbb{N}$ the space of $\mu$-equidistributed sequences: 
\begin{align*}
\mathcal{X}_{\mu}^\mathbb{N} = \left\{(x_n)_{n \in \mathbb{N}} \mid \frac{1}{n}\sum_{i=1}^n \delta_{x_i} \rightharpoonup_{ n \to +\infty } \mu \right\}.
\end{align*}
\medskip

\subsection{Equivalence between small variance and compact homogeneity}
The following theorem characterizes the singular case~\eqref{eq:cov_vanish}, where the variance of $\Phi(\spa_n)$ is at most of order $1 / n^2$ for any polynomial $\spa$. 
\PLM{Let us restate in simpler words our Condition \eqref{eq:cov_vanish}. Given $1\leq k,l \leq p$, suppose that for any $\varphi \in \Ccal_b(\R^{\binom{p}{2}})$, we have
$$0=\kappa(\pi_{k,l},\varphi) = \mathrm{cov}\left(\varphi(d(X_1,\ldots,X_p)),\varphi(d(X_1',\ldots,\overset{(l)}{X_k},\ldots,X_p'))\right).$$
By polarisation, the covariance between any two bounded continuous functions $\psi_1$ and $\psi_2$ of the distances vanishes:
$$\mathrm{cov}\left(\psi_1(d(X_1,\ldots,X_p)),\psi_2(d(X_1',\ldots,\overset{(l)}{X_k},\ldots,X_p'))\right)=0.$$
In particular, taking 
\begin{align*}
\psi_1(d(x_1,\ldots,x_p))&=\varphi(d(x_k,x_1,x_2,\ldots,x_{k-1},x_{k+1},\ldots,x_p));\\
\psi_2(d(x_1,\ldots,x_p))&=\varphi(d(x_l,x_1,x_2,\ldots,x_{l-1},x_{l+1},\ldots,x_p)),
\end{align*}
we obtain
\begin{equation}
\mathrm{cov}\left(\varphi(d(X_1,\ldots,X_p)),\varphi(d(X_1,X_2',\ldots,X_p'))\right)=\kappa(\pi_{1,1},\varphi) =0.\label{eq:cov_vanish2}
\end{equation}
Thus, the vanishing of one kind of covariance $\kappa(\pi_{k,l},\varphi)$ is equivalent to the vanishing of all these covariances for $1 \leq k,l \leq p$, and in the sequel we shall work with the case $k=l=1$.
}
We recall that $\nu$ is the map that associates to any point in $\mathcal{X}$ the law of the random variable $d(x,(X_n)_{n \in \N})$.
\begin{theo}\label{theo:homogeneous}
The following assertions are equivalent:
\begin{enumerate}
\item For all $p \geq 1$ and $\varphi \in \Ccal_b(\mathbb{R}^{\binom{p}{2}})$, $\mathrm{cov}(\varphi(d(X_1,\ldots,X_p)),\varphi(d(X_1,X_2',\ldots,X_p')))=0$.
\item The map $\nu$ is constant.
\item The action of $\mathrm{Isomp}(\spa)$ on $\mathcal{X}$ is transitive.
\item There exists a compact topological group $G$, and $K$ a closed subgroup of $G$ such that 
$$\spaex = (G/K, d_{G/K}, \mu_{G/K}),$$ where $d_{G/K}$ is a distance invariant by the action of $G$ ($d_{G/K}(\overline{gg_1},\overline{gg_2}) = d_{G/K}(\overline{g_1},\overline{g_2})$),
and $\mu_{G/K} = \pi_{*}(\mathrm{Haar}_G)$ is the push-forward of the Haar measure of $G$.
\end{enumerate}
\end{theo}
%

\PLM{\begin{remark}
In the fourth item of Theorem \ref{theo:homogeneous}, the identification of $\spa$ as a compact homogeneous space has to be understood in the space $\mathbb{M}$, that is to say modulo measure-preserving isometries. In particular, one assumes that $\mathcal{X}$ is equal to the support of $\mu$.
\end{remark}}

\begin{proof}
\implication{(1)}{(2)} Let $A$ be a closed subset of $\mathbb{R}^{\binom{p}{2}}$. There exists a sequence $(\varphi_q)_{q \in \mathbb{N}}$ of positive continuous bounded functions converging pointwise to $\mathbbm{1}_{A}$ the indicator function of $A$: take $\varphi_q(x) =  \min(1,1-q \,d(x,A))$. Taking the limit in Equation~\eqref{eq:cov_vanish2} as $q$ goes to infinity, we obtain 
\begin{align*}
&\esper\left[\mathbbm{1}_{A}(d(X_1,X_2,\dots,X_{p}))\mathbbm{1}_{A}(d(X_1,X_2',\dots,X_{p}'))\right] \\
&= \esper\left[\mathbbm{1}_{A}(d(X_1,X_2,\dots,X_{p}))\right] \esper\left[\mathbbm{1}_{A}(d(X_1,X_2',\dots,X_{p}'))\right] .
\end{align*}
If $U = \R^{\binom{p}{2}} \setminus A$, then $1_U = 1-1_A$, so the same is true with $A$ open subset. Let us define the map
\begin{align*}
\mathrm{Ed}_{A} \colon \mathcal{X} &\to \mathbb{R} \\
x &\mapsto \mathbb{E}[\mathbbm{1}_{A}(d(x,X_2,\dots,X_p))].
\end{align*}
We have:
\begin{align*}
\int_{\mathcal{X}} \left(\mathrm{Ed}_{A}\right)^2(x) \, \mu(dx) &= \int_{\mathcal{X}} \mathbb{E}\left[\mathbbm{1}_{A}(d(x,X_2,\dots,X_{p}))\right]\mathbb{E}\left[\mathbbm{1}_{A}(d(x,X_2',\dots,X_{p}'))\right] \,\mu(dx) \\
&= \int_{\mathcal{X}} \mathbb{E}\left[\mathbbm{1}_{A}(d(x,X_2,\dots,X_{p}))\mathbbm{1}_{A}(d(x,X_2',\dots,X_{p}'))\right] \, \mu(dx) \\
&= \esper\left[\int_\mathcal{X} \mathbbm{1}_{A}(d(x,X_2,\dots,X_{p}))\mathbbm{1}_{A}(d(x,X_2',\dots,X_{p}')) \, \mu(dx)\right] \\
&= \esper\left[\mathbbm{1}_{A}(d(X_1,X_2,\dots,X_{p}))\mathbbm{1}_{A}(d(X_1,X_2',\dots,X_{p}'))\right] \\
&= \esper\left[\mathbbm{1}_{A}(d(X_1,X_2,\dots,X_{p}))\right] \esper\left[\mathbbm{1}_{A}(d(X_1,X_2',\dots,X_{p}'))\right] \\
&= \left(\esper\left[\mathbbm{1}_{A}(d(X_1,X_2,\dots,X_{p}))\right]\right)^2 \\
&= \left(\int_{\mathcal{X}} \mathrm{Ed}_{A}(x) \, \mu(dx)\right)^2,
\end{align*}
so the variance $\var(\mathrm{Ed}_{A})$ under $\mu$ vanishes. We have thus showed that
\begin{align*}
\forall A \text{ open set of } \mathbb{R}^{\binom{p}{2}}, \mu\text{-almost surely, }\mathrm{Ed}_{A} \text{ is constant} (= \int_{\mathcal{X}} \mathrm{Ed}_{A}(x) \, \mu(dx)). 
\end{align*}
Fix a countable basis of open subsets $(A_i)_{i \in \N}$ of $\mathbb{R}^{\binom{p}{2}}$. For any $A_{i_1},\dots,A_{i_n}$, there exists a set $\mathcal{X}_{A_{i_1},\dots,A_{i_n}}$ of $\mu$-measure $1$ such that $\mathrm{Ed}_{A_{i_1} \cup \dots \cup A_{i_n}}$ is constant on that set. Hence, there exists a set $\mathcal{X}_0 \subseteq \mathcal{X}$ of $\mu$-measure $1$ such that all the maps $\mathrm{Ed}_{A_{i_1} \cup \dots \cup A_{i_n}}$ are simultaneously constant on $\mathcal{X}_0$. We can replace in the previous statement $\mathcal{X}_0$ by $\mathcal{X}$, because by dominated convergence, $\mathrm{Ed}_{A}$ is continuous over $\mathcal{X}$, and by assumption, $\mathcal{X}$ is the support of $\mu$, that is to say the smallest closed subset with $\mu$-measure $1$. \medskip

Consider now an arbitrary open subset $A \subset \mathcal{X}$, and $x,y \in \mathcal{X}$. We can write $A$ as a union $\bigcup_{i \in I} A_i$, and for any finite subfamily $J \subset I$, we have by assumption
$$\mathrm{Ed}_{\bigcup_{i \in J}A_i} (x) = \mathrm{Ed}_{\bigcup_{i \in J}A_i} (y).$$
By making $J$ grow to $I$, we conclude that $\mathrm{Ed}_A(x)=\mathrm{Ed}_A(y)$. The set of all $A \subset \mathbb{R}^{\binom{p}{2}}$ such that $\mathrm{Ed}_{A}$ is constant is a Dynkin system, so we get that for any Borel subset $A$ of $\mathbb{R}^{\binom{p}{2}}$, the map $\mathrm{Ed}_{A}$ is constant over $\mathcal{X}$. This means that the law of $d(x,X_2,\ldots,X_p)$ is constant over $\mathcal{X}$. As this is true for any $p \geq 1$, and as the measurable structure of $\R^{\mathrm{met}}$ is defined by its finite projections, we conclude that $\nu^x$ does not depend on $x$.

\bigskip
\noindent \implication{(2)}{(1)}
Fix $x_0 \in \mathcal{X}$, and denote $(X_n')_{n \in \mathbb{N}}$ an independent copy of $(X_n)_{n \in \mathbb{N}}$. We can write
\begin{align*}
&\esper[\varphi(d(X_1,X_2,\dots,X_{p}))\,\varphi(d(X_1,X_2',\dots,X_{p}'))] \\ 
&= \esper\left[\int_{\mathcal{X}} \varphi(d(x,X_2,\dots,X_{p}))\varphi(d(x,X_2',\dots,X_{p}'))\, \mu(dx)\right]\\
&= \esper\left[ \varphi(d(x_0,X_2,\dots,X_{p}))\,\varphi(d(x_0,X_2',\dots,X_{p}'))\right]
\\
&= \esper\left[ \varphi(d(x_0,X_2,\dots,X_{p}))\right]\,\esper\left[\varphi(d(x_0,X_2',\dots,X_{p}'))\right]\\
&= \esper\left[\int_{\spa} \varphi(d(x,X_2,\dots,X_{p}))\,\mu(dx)\right]\,\esper\left[\int_{\spa}\varphi(d(x,X_2',\dots,X_{p}'))\,\mu(dx)\right]\\
&= \esper\left[ \varphi(d(X_1,X_2,\dots,X_{p}))\right]\,\esper\left[\varphi(d(X_1,X_2',\dots,X_{p}'))\right]
\end{align*}
because from the second point, the integrals inside the expectations do not depend on $x$.
\bigskip

\noindent \implication{(2)}{(3)} We adapt the arguments of \cite[Section $3\frac{1}{2}$]{gromov2007metric}.
Let $x,y \in \mathcal{X}$, we set $\nu^{eq} = \nu^x = \nu^y$ as the common value of the map $\nu$ by hypothesis.
The law of large numbers gives us $\mu^{\otimes\mathbb{N}} ( \mathcal{X}_\mu^\mathbb{N} ) = 1.$
Then
\begin{align*}
\nu^{eq}\left((\iota^\spa \circ S^x)(\mathcal{X}_\mu^\mathbb{N})\cap (\iota^\spa \circ S^y)(\mathcal{X}_\mu^\mathbb{N}) \right) &=  \nu^x\left((\iota^\spa \circ S^x)\left(\mathcal{X}_\mu^\mathbb{N}\right)\right) + \nu^y\left((\iota^\spa \circ S^y)\left(\mathcal{X}_\mu^\mathbb{N}\right)\right) 
\\
&\quad - \nu^{eq}\left((\iota^\spa \circ S^x)(\mathcal{X}_\mu^\mathbb{N})\cup (\iota^\spa \circ S^y)(\mathcal{X}_\mu^\mathbb{N}) \right)
\\
&\geq \mu^{\otimes \mathbb{N}}(\mathcal{X}_\mu^\mathbb{N}) + \mu^{\otimes \mathbb{N}}(\mathcal{X}_\mu^\mathbb{N}) - 1 = 1.
\end{align*}
It implies the existence of two sequences $(x_n)_{n \in \mathbb{N}}$ and $(y_n)_{n \in \mathbb{N}}$ in $\mathcal{X}^\mathbb{N}$  such that
\begin{itemize}
\item $x_0 = x$ et $y_0 = y$;
\item $(x_n)_{n \in \mathbb{N}}$ et $(y_n)_{n \in \mathbb{N}}$ are in $\mathcal{X}_\mu^\mathbb{N}$;
\item $(d(x_i,x_j))_{i,j} = (d(y_i,y_j))_{i,j}$.
\end{itemize}
By the Portmanteau theorem \cite[Theorem 2.1]{billing}, a $\mu$-equidistributed sequence is dense in the support of $\mu$.
Therefore, there exists a unique isometry $ i \colon \mathcal{X} \to \mathcal{X}$ such that for all $n \in \mathbb{N}$, $i(x_n) = y_n$.
We have for any continuous bounded function $f : \mathcal{X}\to \R$:
\begin{align*}
\frac{1}{n}\sum_{j=1}^n \delta_{x_j}(f \circ i) = \frac{1}{n}\sum_{j=1}^n \delta_{i(x_j)}(f)  = \frac{1}{n}\sum_{j=1}^n \delta_{y_j}(f) .
\end{align*}
By taking the limit of this identity as $n$ goes to infinity, we obtain $\mu(f \circ i) = \mu(f)$. This is true for any $f \in \Ccal_b(\mathcal{X})$, so by \cite[Theorem 1.2]{billing}, $i_*\mu=\mu$. We have therefore constructed $i \in \mathrm{Isomp}(\spa)$ such that $i(x)=y$.
\bigskip

\noindent \implication{(3)}{(2)}
Let $x,y \in \mathcal{X}$, by $3.$, there exists an isometry $i \colon \mathcal{X} \to \mathcal{X}$ with $i(x) = y$ and $i_* \mu = \mu$. We can define $i^\N \colon \mathcal{X}^\N \to \mathcal{X}^\N$ with $i^\N((x_n)_{n \in \N}) = (i(x_n))_{n \in \N}$. We get $(i^\N)_* \mu^{\otimes \N} = \mu^{\otimes \N}$. Let $\varphi \colon \R^{\mathrm{met}} \to \R$ a bounded continuous function, we have with $x_0 = x$ and $y_0 = y$,
\begin{align*}
\int_{\R^{\mathrm{met}}} \varphi(z) \, \nu^x(z) &= \int_{\mathcal{X}^\N} \varphi(d((x_n)_{n \in \N})) \, \mu^{\otimes \N}((x_{n+1})_{n \in \N})
\\
&= \int_{\mathcal{X}^\N} \varphi(d((i(x_n))_{n \in \N}) \, \mu^{\otimes \N}((x_{n+1})_{n \in \N})
\\
&= \int_{\mathcal{X}^\N} \varphi(d((y_n)_{n \in \N}) \, (i^\N)_* \mu^{\otimes \N}((y_{n+1})_{n \in \N})
\\
&= \int_{\mathcal{X}^\N} \varphi(d((y_n)_{n \in \N}) \, \mu^{\otimes \N}((y_{n+1})_{n \in \N})
\\
&=\int_{\R^{\mathrm{met}}} \varphi(z) \, \nu^y(z),
\end{align*}
so $\nu^x=\nu^y$.
\bigskip

\noindent \implication{(4)}{(3)} The action of $G$ on $G/K$ gives rise to translations $(\tau_g)_{g \in G}$ with $\tau_g(\overline{g_1}) = \overline{gg_1}$; they form a subgroup of $\mathrm{Isomp}(G/K)$. For $\overline{g_1}, \overline{g_2} \in G/K$, the translation $\tau_{g_2 g_1^{-1}}$ sends $\overline{g_1}$ to $\overline{g_2}$, so $\mathrm{Isomp}(G/K)$ is transitive on $G/K$.
\bigskip

\noindent \implication{(3)}{(4)} Let $(x_n)_{n \in \mathbb{N}}$ a dense sequence in $\spa$ and
\begin{align*}
D_{\spa, \epsilon} = \left\{ I \subseteq \mathcal{P}(\mathbb{N}) \mid \text{the union }\cup_{n\in I}B(x_n,\epsilon) \text{ is disjoint}\right\};
\end{align*}
this is a poset for the inclusion order, and it is stable by increasing union. We build by induction a maximal element of this set. We set $A_0 = B(x_0,\epsilon)$ and $I_0 = \{0\}$, and then for any $n \in \N$:
\begin{itemize}
\item if $B(x_{n+1},\epsilon)\cap A_n = \emptyset$, then $A_{n+1} = A_n \sqcup B(x_{n+1},\epsilon)$ and $I_{n+1} = I_n \sqcup \{n+1\}$;
\item otherwise, $A_{n+1} = A_n$ and $I_{n+1} = I_n$.
\end{itemize}
Consider $I_{\max} = \bigcup_{n \in \mathbb{N}} I_n$.
\begin{enumerate}
 	\item The set of indices $I_{\max}$ is a maximal element of $\left(D_{\spa, \epsilon}, \subseteq\right)$: 
 	if $n \notin I_{\max}$, then $B(x_{n},\epsilon) \cap A_{n-1}$ is non-empty, and \emph{a fortiori} 
 	$$B(x_n,\epsilon)\cap \left(\bigsqcup_{i \in I_{\max}} B(x_i,\epsilon)\right)\neq \emptyset;$$ 
 	therefore, we cannot add $n$ to $I_{\max}$ and stay in $D_{\spa,\epsilon}$.
 	\item We have $\mathcal{X} = \bigcup_{n \in I_{\max}} B(x_n, 3\epsilon)$. If $x \in \mathcal{X}$, since $(x_n)_{n \in \mathbb{N}}$ is dense in $\mathcal{X}$, there exists $n \in \mathbb{N}$ such that $x \in B(x_n,\epsilon)$. If $n \in I_{\max}$, then obviously $$x \in \bigsqcup_{n \in I_{\max}} B(x_n,\epsilon) \subset \bigcup_{n \in I_{\max}} B(x_n, 3\epsilon),$$
 	 and if $n$ is not in $I_{\max}$, then there exists $n' \in I_{\max}$ such that $ y \in B(x_n,\epsilon)\cap B(x_{n'},\epsilon) \neq \emptyset$. Hence, we have
\begin{align*}
d(x,x_{n'}) \leq d(x,x_n) + d(x_n,y) +d(y,x_{n'}) \leq 3\epsilon.
\end{align*}
\item The set $I_{\max}$ is finite. Indeed, because the action of $\mathrm{Isomp}(\spa)$ over $\spa$ is transitive, the following map is constant:
\begin{align*}
\mathcal{X} &\to \mathbb{R} \\
x &\mapsto \mu(B(x, \epsilon))
\end{align*}
with common value denoted $\mu_{\epsilon}>0$. Consequently,
\begin{align*}
1 \geq \mu\left(\bigsqcup_{n \in I_{\max}}B(x_n,\epsilon)\right) = \sum_{n \in I_{\max}} \mu(B(x_n,\epsilon)) = \mathrm{card}(I_{\max})\,\mu_{\epsilon} 
\end{align*}
because $\mu$ is a probability measure.
 \end{enumerate} 
So, $I_{\max}$ is finite, and we have proved that $\mathcal{X}$ is a pre-compact space. Since $\mathcal{X}$ is complete, $\mathcal{X}$ is compact.
\PLM{ The group of isometries $\mathrm{Isom}(\spa)$ endowed with the compact-open topology defined by the neighborhoods $V(i,K,\epsilon)$ from Equation \eqref{eq:neighborhood} is also a compact Hausdorff space:
 \begin{itemize}
	\item It is a general fact that given two compact metric spaces $\mathcal{X}$ and $\mathcal{Y}$, the space of continuous functions $\Ccal(\mathcal{X},\mathcal{Y})$ endowed with the compact-open topology is metrised by $d(f,g) = \sup_{x \in \mathcal{X}} d(f(x),g(x))$; see \cite[Chapter XII, Section 8]{Dug66}. By restriction, the topology of $\mathrm{Isom}(\mathcal{X})$ is metrisable.
	\item The compactness of $\mathrm{Isom}(\spa)$ is then an immediate application of the Arzela--Ascoli theorem.
\end{itemize}
The subgroup of measure-preserving isometries $\mathrm{Isomp}(\spa)$ is a closed subgroup of $\mathrm{Isom}(\spa)$, hence also compact.}
Since the action of $\mathrm{Isomp}(\spa)$ over $\mathcal{X}$ is transitive, we have $\mathrm{O}_x = \mathcal{X}$ for each $x \in \mathcal{X}$. 
Therefore, we have the following homeomorphism (see \cite[Theorem 2.3.2]{mneimne1986introduction}):
\begin{align*}
\psi : \mathrm{Isomp}(\spa)/\mathrm{St}_x & \to  \mathcal{X}\\ \bar{g} & \mapsto  g \cdot x. 
\end{align*}
Denote $G = \mathrm{Isomp}(\spa)$ and $K=\mathrm{St}_x$, $x$ being an arbitrary reference point in the space $\mathcal{X}$. The homeomorphism $\psi$ allows one to transport the distance $d$ of $\mathcal{X}$  to a $G$-invariant distance $d_{G/K}(\cdot,\cdot)=d(\psi^{-1}(\cdot),\psi^{-1}(\cdot))$, and the measure $\mu$ to a $G$-invariant probability measure $\mu_{G/K}=(\psi^{-1})_*\mu$ on $G/K$.
It remains to prove that $\mu_{G/K} = \pi_*(\mathrm{Haar}_G)$. Given a topological compact Hausdorff space $Z$, we recall the bijective correspondence (see \cite[Chapter IX]{Lang93}):
\begin{alignat*}{2}
\Mcal^1(Z) &\to \{\phi \colon \Ccal(Z,\mathbb{R}) \to \mathbb{R}, \, \mathbb{R}\text{-linear, continuous, positive and with }\phi(1) = 1 \} \\
\mu &\mapsto  \begin{cases} \Ccal(Z,\mathbb{R}) \!\!\!\!&\to \mathbb{R} \\
\qquad f &\mapsto  \int_Z f(z) \mu(dz) .
\end{cases}
\end{alignat*}
To any topological compact Hausdorff group $Z$, we associate the probability Haar measure $\mathrm{Haar}_Z$, and we define
\begin{alignat*}{2}
T \colon \Ccal(G) &\to \Ccal(G/K) \\
f &\mapsto  Tf \colon\begin{cases} G/K \!\!\!\!&\to \mathbb{R} \\
 \,\,\,\, gK \!\!\!\!&\mapsto  \int_K f(gk) \, \mathrm{Haar}_K(dk)
\end{cases}
\end{alignat*}
We denote by $\Ccal(G)_+^*$ the space of positive continuous linear forms on the $\R$-vector space $\Ccal(G)$.
The transformation $T$ induces the contravariant transformation
\begin{align*}
T^* \colon \Ccal(G/K)_+^* &\to  \Ccal(G)_+^* \\
\nu &\mapsto \nu \circ T,
\end{align*}
and any group action $G \times A \to A$ induces the group action
\begin{align*}
G \times \Ccal(A) &\to \Ccal(A)
\\
(g,f) &\mapsto g \cdot f  = \begin{cases} A \!\!\!\!&\to \mathbb{R} \\
 x\!\!\!\! &\mapsto f(g^{-1} \cdot x).
\end{cases}
\end{align*}
Consider the probability measure $\mu_{G/K} $ as an element of $\Ccal(G/K)_+^*$; we have by definition that for any $g \in G$ and $p \in \Ccal(G/K) $, $\mu(g \cdot p) = \mu(p)$. If $q \in \Ccal(G)$, then we have
\begin{align*}
(\mu \circ T)( g \cdot q) = \mu (T(g \cdot q)) &= \mu \left( \begin{cases} G/K\!\!\!\! &\to \mathbb{R} \\
 \,\,\,\,\overline{l} \!\!\!\!&\mapsto  \int_K (g \cdot q) (lk) \, \mathrm{Haar}_K(dk)
\end{cases}\right)
\\
&= \mu \left( \begin{cases} G/K \!\!\!\! &\to \mathbb{R} \\
 \,\,\,\,\overline{l} \!\!\!\!&\mapsto  \int_K q (g^{-1}lk) \, \mathrm{Haar}_K(dk)
\end{cases}\right)
\\
&= \mu \left( g \cdot \begin{cases} G/K \!\!\!\!&\to \mathbb{R} \\
\,\,\,\, \overline{l} \!\!\!\!&\mapsto  \int_K q (lk) \, \mathrm{Haar}_K(dk)
\end{cases}\right)
\\
&= (\mu \circ T)(q).
\end{align*}
so $\mu \circ T = T^*(\mu)$ is the unique $G$-invariant positive normalised continuous linear form on $\Ccal(G)$.
Hence $T^*(\mu) = \mathrm{Haar}_G$, and we finally need to show that $\pi_* \circ T^* = \mathrm{Id}_{\Ccal(G/K)_+^*}$.
However, for any $ \overline{g} \in G/K $ and $f \in \Ccal(G/K)$, $T(f \circ \pi)(\overline{g}) = \int_K f \circ \pi (gk) \,\mathrm{Haar}_K(dk) = \int_K f(\overline{g})\, \mathrm{Haar}_K(dk)  = f(\overline{g})$; 
\PLM{the result follows by functoriality.}
\end{proof}
\medskip

\subsection{Study of the cumulants in the homogeneous case}
\PLM{We now perform the asymptotic analysis of the fluctuations of the observables $\Phi(\spa_n)$ when $\mathcal{X}=G/K$ is a compact homogeneous space. We start by proving an upper bound on the cumulants of $S_n(\varphi,\spa)$ which will be analogue to the one of the method of cumulants, but with different parameters $N_n$ and $D_n$, and with a non-Gaussian limiting distribution; see Theorem \ref{theo:bound_cumulant_homogeneous}.} Our arguments will involve spanning trees of graphs. We recall that a \emph{Cayley tree} of size $r$ is a labeled tree with vertex set $[\![1,r]\!]$; there are $r^{r-2}$ Cayley trees of size $r$. We start with the homogeneous analogue of Proposition \ref{prop:vanish generic}.

\begin{prop}\label{prop:vanish generic homogeneous}
If $\spa$ is a compact homogeneous space, then for $r\geq2$, $\pi \in \mathfrak{Q}(pr)$ and $\varphi \in \Ccal_b(\mathbb{R}^{\binom{p}{2}})$, if $\ell(\pi) > (p-1)r $, then $\kappa(\pi,\varphi) = 0$.
\end{prop}

\begin{proof}
We consider the same trees $(T_A)_{A \in \pi}$, the same graph $G_\pi$ and the same multigraph $H_\pi$ as in the proof of Proposition~\ref{prop:vanish generic}. We have $\sum_{A \in \pi} (|E(T_A)|+1) = pr$. This implies $|E(H_\pi)|=|E(G_\pi)|=\sum_{A \in \pi} |E(T_A)| \leq r-1$ by the assumption on $\ell(\pi)$. If $H_\pi$ is not connected, then the same argument as in Proposition~\ref{prop:vanish generic} gives $\kappa(\pi,\varphi)=0$. Therefore, the only remaining case to treat is when $H_\pi$ is connected and has exactly $r-1$ edges; it is then a Cayley tree. Fix $I=(\ibar^1,\ldots,\ibar^r)$ such that $\mathrm{Sp}_n(I)=\pi$, and an index $k \in [\![1,r]\!]$ which is a leaf of the graph $H_\pi$. By definition of the multigraph $H_\pi$, the block of indices $\ibar^k$ shares exactly one index with all the other blocks $\ibar^{j \neq k}$:
$$\mathrm{Card}\left(\overline{\mathrm{Im}}(\ibar^k)\cap \bigcup_{j\neq k} \overline{\mathrm{Im}}(\ibar^j)\right) =1. $$
To fix the ideas, let us assume that $k=1$ and that the shared index in $\ibar^1=(\ibar^1_1,\ldots,\ibar^1_p)$ is the first one. To compute the cumulant, we consider
\begin{align*}
\esper\left[\E^{t_1 \varphi(d(X_{\ibar^1}))+\cdots+ t_r \varphi(d(X_{\ibar^r}))}\right] = \int_{\mathcal{X}}\left(\int_{\mathcal{X}^{(p-1)r}} \E^{t_1 \varphi(d(x_{\ibar^1}))+\cdots+ t_r \varphi(d(x_{\ibar^r}))}\, \mu^{\otimes (p-1)r}((x_i)_{i \neq \ibar^1_1})\right)\mu(dx_{\ibar^1_1}).
\end{align*}
Denote $F(x_{\ibar^1_1})$ the integral where one has integrated all the variables except $x_{\ibar^1_1}$. If $x_0$ is an arbitrary point in $\mathcal{X}$, then for any $x_{\ibar^1_1}$, we have an isometry $\psi \in\mathrm{Isomp}(\spa)$ such that $\psi(x_{\ibar^1_1}) = x_0$, because $\spa$ is homogeneous.
So, denoting $y_a = \psi(x_a)$, we get 
\begin{align*}
F(x_{\ibar^1_1})&=\int_{\mathcal{X}^{(p-1)r}} \E^{t_1 \varphi(d(x_{\ibar^1}))+\cdots+ t_r \varphi(d(x_{\ibar^r}))}\, \mu^{\otimes (p-1)r}((x_i)_{i \neq \ibar^1_1}) \\
&= \int_{\mathcal{X}^{(p-1)r}} \E^{t_1 \varphi(d(y_{\ibar^1})))+\cdots+ t_r \varphi(d(y_{\ibar^r})))}\, \mu^{\otimes (p-1)r}((y_i)_{i \neq \ibar^1_1})\\ 
&= F(x_0) 
\end{align*}
and the integral does not depend on $x_{\ibar^1_1}$.
So, we have
\begin{align*}
\esper\left[\E^{t_1 \varphi(d(X_{\ibar^1}))+\cdots+ t_r \varphi(d(X_{\ibar^r}))}\right] = F(x_0) &= \int_{\mathcal{X}} F(x_0)\,\mu(dx_0)\\
&= \int_{\mathcal{X}^{(p-1)r+1}} \E^{t_1 \varphi(d(x_{\jbar^1}))+\cdots+ t_r \varphi(d(x_{\jbar^r}))}\, \mu^{\otimes (p-1)r+1}((x_j)_{j \in \overline{\mathrm{Im}}(\jbar)})
\end{align*}
where $J=(\jbar^1,\ldots,\jbar^r)$ is the same collection of indices as $I$, except that we have replaced $\ibar^1_1$ by a new index different from all the other indices. In this new collection, $\jbar^1$ does not share any index with the other families $\jbar^{k\geq 2}$, so $X_{\jbar^1}$ is independent from the other variables, and
$$\esper\left[\E^{t_1 \varphi(d(X_{\ibar^1}))+\cdots+ t_r \varphi(d(X_{\ibar^r}))}\right]= \esper\left[\E^{t_1 \varphi(d(X_{\ibar^1}))}\right] \esper\left[\E^{t_2\varphi(d(X_{\ibar^2}))+\cdots+ t_r \varphi(d(X_{\ibar^r}))}\right].$$
Looking at the coefficient of $[t_1\cdots t_r]$ in the logarithm of the Laplace transform, we conclude that the joint cumulant vanishes.
\end{proof}

\PLM{\begin{remark}
The proof of this proposition leads to a slightly stronger result: if $\pi \in \mathfrak{Q}(pr)$ is a set partition such that $H_\pi$ is disconnected or is a tree, or even is a connected graph with one vertex of valence $1$, then the corresponding cumulant vanishes. For instance, with $r=2$ and $p=6$, the following set partition
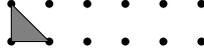
\vspace{2mm}
\begin{center}
\begin{tikzpicture}[scale=0.5]
\foreach \x in {0,1,2,3,4,5}
{\fill (\x,0) circle (3pt);
\fill (\x,1) circle (3pt);}
\fill[gray] (1,0) -- (0,0) -- (0,1) -- (1,0);
\draw (1,0) -- (0,0) -- (0,1) -- (1,0);
\end{tikzpicture}
\end{center}
\vspace{2mm}

\noindent which identifies one index of the first block of indices $\ibar^1$ with two distinct indices of the second block $\ibar^2$ satisfies $\ell(\pi) = 10 = (p-1)r$, but the corresponding graph $H_\pi$ is the unique Cayley tree on $2$ vertices, so $\kappa(\pi,\varphi)=0$ for any function $\varphi \in \Ccal(\R^{\binom{p}{2}})$. The most general condition which leads to the vanishing of the joint cumulant $\kappa(\pi,\varphi)=0$ is the following: if there exists an integer $k \in [\![1,r]\!]$ such that, among the integers $(k-1)p+1,\ldots,kp$, the set partition $\pi \in \mathfrak{Q}(pr)$ contains $p-1$ singletons (and the remaining integer of this block which can be connected to many other integers in the other blocks), then $\kappa(\pi,\varphi)=0$. Indeed, we can then use the same trick as above to replace in the computation of the joint Laplace transform the family $\ibar^k$ by a family of indices $\jbar^k$ which are all distinct and which are not shared by the other families $\ibar^{a \neq k}$. We call such a set partition $\pi$ \emph{homogeneously vanishing}.
\end{remark}}
\medskip

In the homogeneous case, the variance $ \var(S_{n}(\varphi,\spa))$ is a polynomial function of degree smaller than $2(p-1)$. We have
\begin{align*}
\kappa^{(2)}(S_n(\varphi,\spa)) = \sum\limits_{\substack{\pi \in \mathfrak{Q}(2p) \\ \ell(\pi) \leq 2(p-1)}}\kappa(\pi,\varphi) \,n^{\downarrow \ell(\pi)}. 
\end{align*}
By using the previous remark, we can identify the set partitions with $\ell(\pi)=2(p-1)$ and $\kappa(\pi,\varphi) \neq 0$. For $1 \leq k_1,l_1,k_2,l_2 \leq p$ with $k_1 \neq k_2, l_1 \neq l_2$, we define the set partition
\begin{align*}
\pi_{k_1,l_1,k_2,l_2} = \{\{k_1,l_1\},\{k_2,l_2\},\{t\} \cup \{ \{t\} \, ; \, t \in [\![1,2p]\!] \setminus \{k_1,k_2,l_1+p,l_2+p\} \}.
\end{align*}
Then we have  the following equality (the bracket is the extraction of the coefficient of degree $n^{2(p-1)}$ in the polynomial in the variable $n$): 
\begin{equation}
\kappa^{(2)}(S_n(\varphi,\spa))[n^{2(p-1)}] = \sum\limits_{\substack{1 \leq k_1,l_1,k_2,l_2 \leq p \\ k_1 <k_2,\,\, l_1 \neq l_2}} \kappa(\pi_{k_1,l_1,k_2,l_2},\varphi) \vcentcolon= \sigma_{\text{hom}}^2.\label{eq:sigma_hom}
\end{equation}
\begin{prop}\label{prop:limit_cumulant}
Suppose that $\sigma_{\text{hom}}^2 > 0$. If $Y_n(\varphi,\spa) = \frac{\Phi(\spa_n) - \Phi(\spa)}{\sqrt{\var(\Phi(\spa_n))}} $, then we have convergence of all the cumulants of these variables: for any $r \geq 1$, there exists $a_r \in \R$ such that
$$\kappa^{(r)}(Y_n(\varphi,\spa)) {\longrightarrow}_{n \to +\infty } a_r .$$
\end{prop}

\begin{proof}
For $r=1$, $\kappa^{(r)}(Y_n(\varphi,\spa)) = 0$ and for $ r\geq 2$
\begin{align*}
\kappa^{(r)}(Y_n(\varphi,\spa)) &= \kappa^{(r)}\left(\frac{\Phi(\spa_n) - \Phi(\spa)}{\sqrt{\var(\Phi(\spa_n))}}\right) \\
&= \kappa^{(r)}\left(\frac{S_{n}(\varphi,\spa) - \esper[S_{n}(\varphi,\spa)]}{\sqrt{\var(S_{n}(\varphi,\spa))}}\right) = \frac{\kappa^{(r)}(S_{n}(\varphi,\spa))}{\left( \var(S_{n}(\varphi,\spa))\right)^{r/2}} \,.
\end{align*}
We know that for each $r \geq 2$, $\kappa^{(r)}(S_{n}(\varphi,\spa))$ is a polynomial function of degree less than $(p-1)r$, according to Propositions~\ref{prop:polynomiality} and \ref{prop:vanish generic homogeneous}. We can write $\kappa^{(r)}(S_{n}(\varphi,\spa)) = V(n) = \sum_{i=0}^{(p-1)r} v_i n^i$ and $\var(S_{n}(\varphi,\spa)) = \kappa^{(2)}(S_{n}(\varphi,\spa)) = W(n) = \sum_{i=0}^{2(p-1)} w_i n^i$; the assumption $\sigma_{\text{hom}}^2>0$ amounts to $w_{2(p-1)} > 0$. So we have
\begin{align*}
\lim\limits_{n \to +\infty} \kappa^{(r)}(Y_n(\varphi,\spa)) &= \lim\limits_{n \to +\infty} \frac{v_{(p-1)r}\, n^{(p-1)r}}{ \left(w_{2(p-1)}\, n^{2(p-1)}\right)^{r/2}} = \frac{v_{(p-1)r}}{(w_{2(p-1)})^{r/2}} = a_r.\qedhere
\end{align*}
\end{proof}
\medskip

\PLM{The following theorem ensures that the $a_r$'s are not too large, so that we can sum them and obtain the Laplace transform of a limiting distribution of $Y_n(\varphi,\spa)$.}

\begin{theo}\label{theo:bound_cumulant_homogeneous}
In the case where $\spa$ is a compact homogeneous space, we have for any $\varphi \in \Ccal(\R^{\binom{p}{2}})$ and any $r\geq 2$ the upper bound 
$$|\kappa^{(r)}(S_n(\varphi,\spa))| \leq (Ap^2)^r \,(2r)^{r-1} \, n^{(p-1)r}$$
with $A=\|\varphi\|_\infty$.
\end{theo}
\begin{proof}
We are going to adapt the proof of the upper bound~\eqref{bound:cumulant} which can be found
in \cite[Chapter 9]{feray2016mod}. We expand by multilinearity the cumulant and we start by controlling each term of the following sum: 
\begin{align*}
\kappa^{(r)}(S_n(\varphi,\spa)) = \sum_{(\ibar^1,\dots,\ibar^r) \in V^r} \kappa\left(\varphi(d(X_{\ibar^1})),\dots,\varphi(d(X_{\ibar^r}))\right),
\end{align*}
with $V=[\![1,n]\!]^p$. With $A = \| \varphi\|_{\infty} $, Equation (9.9) in \cite{feray2016mod}  gives 
\begin{align*}
|\kappa\left(\varphi(d(X_{\ibar^1})),\dots,\varphi(d(X_{\ibar^r}))\right)| \leq A^r 2^{r-1}\, \mathrm{ST}(H_\pi),
\end{align*}
where $\pi = \mathrm{Sp}_n(\ibar^1,\ldots,\ibar^r)$ and $\mathrm{ST}(H_\pi)$ is the number of spanning trees of the multigraph $H_\pi$.
Now, we have identified in a previous remark the cumulants $\kappa(\pi,\varphi)$ which vanish in the homogeneous case, so we can add this condition to the upper bound. Thereby, we have
\begin{align*}
|\kappa(X_{\ibar^1},\dots,X_{\ibar^r})| \leq A^r 2^{r-1}\, \mathrm{ST}(H_\pi)\, \mathbbm{1}_{\mathrm{NHV}(\pi)},
\end{align*}
where $\mathrm{NHV}(\pi)$ is the condition "$\pi$ is not homogeneously vanishing".
Summing over $V^r$, we get by using the triangle inequality
\begin{align*}
|\kappa^{(r)}(S_{n}(\varphi,\spa))| &\leq A^r 2^{r-1} \sum_{\ibar^1 \in V} \left[ \sum_{(\ibar^2,\dots,\ibar^r) \in V^{r-1}}\mathrm{ST}(H_{\mathrm{Sp}_n(I)}) \,\mathbbm{1}_{\mathrm{NHV}(\mathrm{Sp}_n(I))} \right]
\\
&\leq A^r 2^{r-1}  \sum_{\ibar^1 \in V} \left[\sum_{T\text{ Cayley tree of size }r} \sum_{(\ibar^2,\dots,\ibar^r) \in V^{r-1}}\mathbbm{1}_{T \subset H_{\mathrm{Sp}_n(I)}} \,\mathbbm{1}_{\mathrm{NHV}(\mathrm{Sp}_n(I))} \right].
\end{align*}
Now, we can bound the expression in the bracket by adapting the Lemma 9.3.5 in \cite{feray2016mod} to the homogeneous case. Indeed, let us fix a Cayley tree $T$ of size $r$ and an element $\ibar^1 \in V$. The lists $(\ibar^2,\dots,\ibar^r)$ which have a non-zero contribution in the sum 
$$\sum_{(\ibar^2,\dots,\ibar^r) \in V^{r-1}}\mathbbm{1}_{T \subset H_{\mathrm{Sp}_n(I)}} \,\mathbbm{1}_{\mathrm{NHV}(\mathrm{Sp}_n(I))}$$
 are constructed as follows. We fix a vertex $k\neq 1$ of degree one (a leaf) in $T$, and we shall choose $\ibar^k$ at the end. Before that:
 \begin{itemize}
  	\item We start by choosing the $\ibar^j$'s with $j$ neighbour of $1$ in $T$ and $j \neq k$. For each such family, $\ibar^1$ and $\ibar^j$ share at least one index, so the number of possibilities for $\ibar^j$ is smaller than $D_n=p^2\,n^{p-1}$.
  	\item We pursue the construction with the neighbours of the neighbours of $1$, and so on but leaving always on the side the vertex $k$. Each time, there are at most $p^2\,n^{p-1}$ possibilities for $\ibar^j$. Moreover, as $k$ is a leaf of $T$, our inductive construction enumerates all the vertices in $[\![1,r]\!]$ but $k$.
  \end{itemize} 
  We therefore have less than $(p^2\,n^{p-1})^{r-2}$ possibilities for $(\ibar^2,\dots,\ibar^r)\setminus \{\ibar^k\}$. We finally choose $\ibar^k$, using now the fact that if the list $(\ibar^2,\dots,\ibar^r)$ yields a non-zero contribution, then $\pi$ is not homogeneously vanishing and $\ibar^k$ must share at least two \emph{distinct} indices with other families $\ibar^a$ and $\ibar^b$ (we may have $a=b$). Consequently, there are less than 
  $$p^4\,(r-1)\,n^{p-2}$$ possible values for $\ibar^k \in V$: one family $\ibar^a$ is obtained by taking the unique neighbour $a$ of $k$ in $T$, there are $(r-1)$ possibilities for the other family $\ibar^b$, then $p^4$ possibilities for the choices of positions of indices that are shared, and $n^{p-2}$ possibilities for the other indices in the family $\ibar^k$. So,
$$\sum_{(\ibar^2,\dots,\ibar^r) \in V^{r-1}}\mathbbm{1}_{T \subset H_{\mathrm{Sp}_n(I)}} \,\mathbbm{1}_{\mathrm{NHV}(\mathrm{Sp}_n(I))} \leq (p^2\,n^{p-1})^{r-2}\,p^4\,(r-1)\,n^{p-2}\leq p^{2r}\,r\,n^{pr-p-r}.$$
  As there are $r^{r-2}$ Cayley trees of size $r$, and $n^p$ possibilities for $\ibar^1$, we finally get the upper bound
\begin{equation*}
|\kappa^{(r)}(S_{n}(\varphi,\spa))| \leq A^r\, 2^{r-1} \,r^{r-1}\,p^{2r}\, n^{(p-1)r}.\qedhere
\end{equation*}
\end{proof}
\medskip

\subsection{Central limit theorem for the homogeneous case}
We can finally prove the analogue of Theorem \ref{theo:generic_case} when $\spa$ is a compact homogeneous space.
\begin{theo}[Fluctuations in the homogeneous case]\label{theo:singular_case}
Let $\spa$ be a compact homogeneous space, $\varphi \in \Ccal(\R^{\binom{p}{2}})$ and $\Phi=\Phi^{p,\varphi}$. Suppose that $\sigma_{\text{hom}}^2(\varphi,\spa) = \lim_{n \to \infty}\frac{\var(S_n(\varphi,\spa))}{n^{2(p-1)}} > 0$; a combinatorial expansion of $\sigma_{\text{hom}}^2(\varphi,\spa)$ is provided by Equation \eqref{eq:sigma_hom}. Then, the sequence 
$$Y_n(\varphi,\spa) = \frac{\Phi(\spa_n) - \Phi(\spa)}{\sqrt{\var(\Phi(\spa_n))}}$$ converges in distribution toward a real-valued random variable $Y(\varphi,\spa)$ having for cumulants the sequence 
$$\kappa^{(r)}(Y(\varphi,\spa)) = a_r = \frac{1}{(\sigma_{\text{hom}})^{\frac{r}{2}}}\,\sum_{\substack{\pi \in \mathfrak{Q}(pr) \\ \ell(\pi) = (p-1)r \\ \pi \text{ non homogeneously vanishing}}} \kappa(\pi,\varphi),$$ 
where $\kappa(\pi,\varphi)$ is defined by Equation \eqref{eq:kappa_pi}.
The law of the limit $Y(\varphi,\spa)$ is determined by these cumulants $(a_r)_{r \geq 1}$. Under the assumption $\sigma_{\text{hom}}(\varphi,\spa)>0$, the renormalisation $\sqrt{\var(\Phi(\spa_n))}$ is of order $n^{-1}$, and more precisely, $\var(\Phi(\spa_n))$ is a polynomial in $n^{-1}$ without constant term and without term $\alpha\, n^{-1}$; its leading term is $\sigma_{\text{hom}}^2(\varphi,\spa)\,n^{-2}$.

\end{theo}

\begin{proof}
Theorem \ref{theo:bound_cumulant_homogeneous} shows that for any $n \in \N$, the log-Laplace transform $\log \esper[\E^{z Y_n}]$ is absolutely convergent on a fixed disc of radius $R>0$, with $R$ independent of $n$. Indeed, denoting $\var(S_n) = (\sigma_{n,\text{hom}})^2\,n^{2(p-1)}$, we obtain by using Stirling's estimate
\begin{align*}
\sum_{r=2}^\infty \frac{|\kappa^{(r)}(Y_n)|}{r!} |z|^r \leq \sum_{r=2}^\infty \frac{(Ap^2\E)^r (2r)^{r-1}}{r^r}\left(\frac{|z|}{\sigma_{n,\text{hom}}}\right)^r \leq \sum_{r=2}^\infty \left(\frac{2|z|Ap^2\E}{\sigma_{n,\text{hom}}}\right)^r.
\end{align*}
Since $\sigma_{n,\text{hom}} \to \sigma_{\text{hom}} >0$, we see that for $n$ large enough, if
$$|z| \leq R = \frac{\sigma_{\text{hom}}}{10Ap^2},$$
then $\log \esper[\E^{z Y_n}]$ is convergent and uniformly bounded on this disk. Taking the exponentials, the same is true for the Laplace transforms $\esper[\E^{z Y_n}]$, and by Proposition \ref{prop:limit_cumulant}, these holomorphic functions converge uniformly on $D(0,R)$ towards 
$$\exp\left(\sum_{r=2}^\infty \frac{a_r}{r!}\,z^r\right) = \lim_{n \to \infty} \esper[\E^{zY_n}].$$
By standard arguments (see \cite[p.~390]{Bil95}), this implies the convergence in law towards a random variable $Y$ whose moment-generating function $\esper[\E^{zY}]$ is the left-hand side of the equation above. Since this Laplace transform is convergent on a disc with positive radius, $Y$ is determined by its moments.
\end{proof}

Let us compare Theorems \ref{theo:generic_case} and Theorems \ref{theo:singular_case}. In the generic case, the variance of $\Phi(\spa_n)$ is expected to be of order
$$O\left(\frac{n^{2p-1}}{(n^p)^2}\right) = O\left(\frac{1}{n}\right),$$
so the fluctuations of $\Phi(\spa_n)$ are usually of order $O(n^{-1/2})$, and asymptotically (mod-)Gaussian. By usually we mean that one specific observable $\varphi$ might satisfy "by chance" $\sigma(\varphi,\spa)=0$, but this is in general not the case; and by Theorem \ref{theo:homogeneous} the vanishing of all these limiting variances is almost equivalent to $\spa$ being compact homogeneous (the \emph{almost} is related to the replacement of Condition \eqref{eq:singular_case} by the simpler Condition \eqref{eq:cov_vanish}; they might be equivalent). In the homogeneous case, the variance of $\Phi(\spa_n)$ is expected to be of order
$$O\left(\frac{n^{2p-2}}{(n^p)^2}\right) = O\left(\frac{1}{n^2}\right),$$
so the fluctuations of $\Phi(\spa_n)$ are now of order $O(n^{-1})$. What remains to be seen is that our estimates on cumulants in the homogeneous case are in a sense optimal: we have the best possible upper bound for these cumulants, and in particular we can have $a_{r \geq 3} \neq 0$, whence a non-Gaussian limiting distribution. The last section of the paper is devoted to the analysis of one such example.
\medskip

\subsection{Concentration inequalities}
Since the cumulant estimate from Theorem \ref{theo:bound_cumulant_homogeneous} holds for any $n$, we can use it in combination with Chernoff's inequality in order to obtain:
\begin{prop}\label{prop:chernoff}
Let $\spa$ be a compact homogeneous space, and $\phi \in \mathcal(\R^{\binom{p}{2}})$ such that $A =\|\varphi\|_\infty$ and $\sigma_{\mathrm{hom}}^2(\varphi,\spa)>0$. We denote as above $\sigma_{n,\mathrm{hom}}^2(\varphi,\spa) = \frac{\var(S_n(\varphi,\spa))}{n^{2(p-1)}}$, and
$$q_n = \frac{2Ap^2}{\sigma_{n,\hom}} \geq 1.$$ 
For any $x \geq 0$ and any $n$,
$$\proba\!\left[|Y_n(\varphi,\spa)|\geq \frac{q_n x}{\E}\right] \leq 2\,\exp\left(\frac{\log(1+x)-x}{\E^2}\right).$$
The same estimate holds with $Y_n(\varphi,\spa)$ replaced by its limit in distribution $Y(\varphi,\spa)$, and $q_n$ replaced by its limit $q$.
\end{prop}
\begin{proof}
Note that the case $r=2$ of Theorem \ref{theo:bound_cumulant_homogeneous} yields $q_n =  \frac{2Ap^2}{\sigma_{n,\hom}} \geq 1$ for any $n$. By Chernoff's inequality and by using Stirling's estimate to get rid of the factorials, we obtain for any $t,x \geq 0$
\begin{align*}
\proba[Y_n(\varphi,\spa) \geq x] &\leq \exp\left(-tx + \sum_{r=2}^\infty \frac{|\kappa^{(r)}(S_n(\varphi,\spa)|}{r!}\left(\frac{t}{\sigma_{n,\hom}\,n^{(p-1)}}\right)^{\!r}\right) \\
&\leq \exp\left(-tx + \frac{1}{\E^2}\sum_{r=2}^\infty \frac{1}{r}\,z^r\right) = \exp\left(-tx -\frac{1}{\E^2} \log(1-z) - \frac{z}{\E^2}\right)
\end{align*}
where 
$$z = \frac{2Ap^2}{\sigma_{n,\hom}}\,\E\, t=q_n\,\E \,t$$ 
is supposed strictly smaller than $1$, so that the power series on the second line is convergent.  The optimal value of $t$ in terms of $x$ is given by the equations
$$x= \frac{q_n^2 t}{1-q_n \E t} \qquad;\qquad z=\frac{q_n\E x}{q_n^2+q_n\E x}\qquad;\qquad t = \frac{x}{q_n^2 + q_n\E x}.$$
This choice of $t$ yields
$$\proba[Y_n(\varphi,\spa) \geq x] \leq \exp\left(\frac{1}{\E^2}\left(\log\left(1+\frac{\E x}{q_n}\right)-\frac{\E x}{q_n}\right)\right).$$
We obtain a two-sided upper bound on the tail of the distribution of $|Y_n(\varphi,\spa)|$ by replacing $Y_n(\varphi,\spa)$ by $-Y_n(\varphi,\spa)$, which satisfies the same hypotheses. Finally, the same arguments hold with $Y(\varphi,\spa)$ replaced by $Y_n(\varphi,\spa)$, since we have convergence in law and in moments.
\end{proof}

\begin{remark}
One can wonder whether there exists in the homogeneous case a Berry--Esseen upper bound similar to the one from Theorem \ref{theo:generic_case}. We believe that the approach from \cite{feray2017mod} cannot be used here, for two reasons:
\begin{itemize}
	\item The concentration inequality stated above is the only thing about the limiting distribution of the variables $Y_n(\varphi,\spa)$ that we able to prove with the techniques of this paper. In particular, we do not know whether this limiting distribution is discrete or absolutely continuous with respect to the Lebesgue measure. This prohibits the use of the inequality from \cite[Chapter XVI, Equation (3.13)]{Fel71}, which is the starting point of the Fourier approach to Berry--Esseen bounds.
	\item Besides, we do not have a large \emph{zone of control} on the Fourier transform of $Y_n(\varphi,\spa)$ as in \cite{feray2017mod}; the upper bound on the cumulants yields an upper bound on the Fourier transform $\esper[\E^{\I \xi Y_n(\varphi,\spa)}]$ on a zone of size $O(1)$, but it seems difficult to extend it to a larger zone, which is a requirement in order to obtain a meaningful upper bound on the Kolmogorov distance $d_{\mathrm{Kol}}(Y_n(\varphi,\spa),Y(\varphi,\spa))$.
\end{itemize}
The study of the Cauchy transform of the variables $Y_n(\varphi,\spa)$ (instead of the Fourier and Laplace transforms) might lead to a solution of the first problem. 
\end{remark}

\bigskip

\section{Sample model for the circle and a non-Gaussian limit}\label{sec:circle}
Throughout this section, $\Phi$ is the observable of metric measure spaces with degree $3$ associated to the continuous bounded function
$$\varphi(d(x,y,z)) = \min(1,d(x,y))\times \min(1,d(y,z)).$$
In particular, if $\spa=\spaex$ is a metric measure space with diameter smaller than $1$, then
$$\Phi(\spa) = \int_{\mathcal{X}^3} d(x,y)\, d(y,z)\,\mu^{\otimes 3}(dx\,dy\,dz).$$
Let us consider the metric measure space $\mathcal{X} = \R/\Z$. For $x \in \R$, we denote $\overline{x}$ the class of $x$ modulo $1$. The space $\mathcal{X}$ is endowed with the geodesic distance
$$d(\overline{x},\overline{y}) = \inf_{k \in \Z} |x-y-k|$$
and with the projection $\mu$ of the Lebesgue measure, which is a probability measure. It is obviously a compact homogeneous space in the sense of Section \ref{sec:homogeneous_case}, and even a compact Lie group. Therefore, by Theorem \ref{theo:singular_case}, if $\spa_n$ is the sample model of order $n$ associated to this space, then
$$Y_n(\varphi,\spa) =\frac{\Phi(\spa_n)-\Phi(\spa)}{\sqrt{\var(\Phi(\spa_n))}}$$
converges towards a limiting distribution, assuming that 
$$ n^2\,\var(\Phi(\spa_n)) = \frac{\var(S_n(\varphi,\spa))}{n^4}$$
admits a strictly positive limit $\sigma^2_{\mathrm{hom}}$. The objective of this section is to prove that this limiting distribution indeed exists and \emph{is not} the Gaussian distribution. To this purpose, we shall compute the three first cumulants of $S_n(\varphi,\spa)$, and prove in particular that $\kappa^{(3)}(Y_n(\varphi,\spa))$ admits a non-zero limit. 

\begin{remark}
The observable that we have chosen is not the simplest counterexample to the asymptotic normality: we could have considered the degree $2$ observable $\varphi'(d(x,y)) = \min(1,d(x,y))$. Our choice of the degree $3$ observable $\varphi$ enables us to explain how to compute the moments and cumulants of a general observable $\Phi(\spa_n)$ (we believe that the explanations are a bit clearer with an example larger than the smallest possible one).
\end{remark}

\subsection{Graph expansion of the moments of monomial observables}
Let us consider in full generality the \emph{monomial observables} $M_G$ attached to multigraphs. Let $G$ be a (unoriented) graph on $p$ vertices $1,2,\ldots,p$, possibly with loops and with multiple edges. We associate to $G=(V,E)$ and to a metric measure space $\spa=(\mathcal{X},d,\mu)$ the function 
\begin{align*}
F_G : \mathcal{X}^p &\to \R_+ \\
(x_1,\ldots,x_p) &\mapsto \prod_{\{i,j\}\,=\, e \in E} \min(1,d(x_i,x_j)).
\end{align*}
For instance, the function $\varphi$ introduced above is
$\varphi(d(x_1,x_2,x_3)) = F_G(x_1,x_2,x_3)$ with
$$G = \begin{tikzpicture}[scale=1,baseline=-1.5mm]
\draw (0,0) -- (2,0) ;
\foreach \x in {0,1,2}
{\fill [color=white] (\x,0) circle (0.2);
\draw (\x,0) circle (0.2);}
\draw (0,0) node {$1$};
\draw (1,0) node {$2$};
\draw (2,0) node {$3$};
\end{tikzpicture}.$$
We denote $M_G(\spa) = \int_{\mathcal{X}^{p}} F_G(x_1,\ldots,x_p)\,\mu^{\otimes p}(dx_1\cdots dx_p)$. This quantity is a polynomial observable of $\spa$, and it only depends on the unlabeled graph underlying $G$. The following proposition relates these observables and the moments of the random functions $M_{G}(\spa_n)$.
\begin{prop}\label{prop:moments_monomial_observables}
Fix a multigraph $G$ on $p$ vertices, and a metric measure space $\spa$, with sample model $\spa_n$ for all order $n$. For any $r\geq 1$, we have:
$$ \esper[(M_G(\spa_n))^r] = \frac{1}{n^{pr}} \sum_{\pi \in \mathfrak{Q}(pr)} n^{\downarrow \ell(\pi)} \,M_{G^r \downarrow \pi}(\spa),$$
where $G^r$ denotes the disjoint union of $r$ copies of $G$, and $G^r \downarrow \pi$ is the contraction of this graph according to a set partition $\pi$.
\end{prop}

By contraction of a multigraph $H$ according to a set partition $\pi$ of its vertex set $V$, we mean the multigraph $H \downarrow \pi$ whose vertices are the parts of $\pi$, and where every edge $\{a,b\}$ of the original graph $H$ becomes an edge between the parts $\pi(a)$ and $\pi(b)$ containing respectively $a$ and $b$ (and a loop if $\pi(a)=\pi(b)$).

\begin{proof}
By definition, if $(X_n)_{n \geq 1}$ is a sequence of independent variables distributed according to $\mu$ and with sequence of empirical measures $(\mu_n)_{n \geq 1}$, then
\begin{align*}
M_G(\spa_n) &= \int_{\mathcal{X}^p} F_G(x_1,\ldots,x_p)\,(\mu_n)^{\otimes p}(dx_1\cdots dx_p) = \frac{1}{n^p} \sum_{1\leq i_1,\ldots,i_p \leq n} F_G(X_{i_1},\ldots,X_{i_p}).
\end{align*}
We denote as usual $\ibar$ an arbitrary family of $p$ indices $i_1,\ldots,i_p$. Given $I=(\ibar^1,\ldots,\ibar^r)$, if $\pi=\mathrm{Sp}_n(I)$ is the set partition of $[\![ 1,pr ]\!] = [\![1,p]\!]^r$ whose parts correspond to the sets of equal indices in $I$, then we have
$$\prod_{a=1}^r F_G(X_{i^a_1},\ldots,X_{i^a_p}) =_{(\mathrm{distribution})} F_{G^r \downarrow \pi}(X_1,\ldots,X_{\ell(\pi)}).$$
Indeed, if one chooses for every part $\pi_c$ of $\pi$ an index $i^{a_c}_{b_c}$ falling in this part, then one has the identity
$$\prod_{a=1}^r F_G(X_{i^a_1},\ldots,X_{i^a_p}) = F_{G^r \downarrow \pi}\!\left(X_{i^{a_1}_{b_1}},\ldots,
X_{i^{a_{\ell(\pi)}}_{b_{\ell(\pi)}}}\right),$$
and the variables $X_{i^{a_c}_{b_c}}$ are all distinct by definition of $\pi$; the identity in distribution follows by a relabeling of these variables.\medskip

We therefore have:
\begin{align*}
\esper[(M_G(\spa_n))^r] 
&=\frac{1}{n^{pr}} \sum_{I=(\ibar^1,\ldots,\ibar^r) \in [\![ 1,n]\!]^{pr}} \esper[F_{G^r \downarrow \mathrm{Sp}_n(I)}(X_1,\ldots,X_{\ell(\mathrm{Sp}_n(I))})] \\
&= \frac{1}{n^{pr}} \sum_{I=(\ibar^1,\ldots,\ibar^r) \in [\![ 1,n ]\!]^{pr}} M_{G^r \downarrow \mathrm{Sp}_n(I)}(\spa),
\end{align*}
and the result follows by gathering the list of indices $I$ according to their set partitions $\mathrm{Sp}_n(I)$.
\end{proof}

\begin{exa}
Let $G$ be the graph on $3$ vertices previously introduced, and $r=2$. Note that if a graph $H= G^2 \downarrow \pi$ contains a loop, then the corresponding monomial $F_H$ vanishes on $\mathcal{X}^{|V(H)|}$. There are $203$ set-partitions of size $6$, but only $67$ of them yield a graph  $H= G^2 \downarrow \pi$ without loop. Gathering these graphs according to their isomorphism types, we obtain:
\begin{align*}
&n^6\,\esper[(M_{\begin{tikzpicture}[scale=0.3]
\draw (0,0) -- (2,0);
\foreach \x in {(0,0),(1,0),(2,0)}
\fill \x circle (5pt);
\end{tikzpicture}}(\spa_n))^2]\\
& = 8\left(
n^{\downarrow 4}\,M_{\begin{tikzpicture}[scale=0.3]
\draw (2,-0.5) -- (1,0) -- (2,0.5);
\draw (1,0) .. controls (0.7,0.3) and (0.3,0.3) .. (0,0) ;
\draw (1,0) .. controls (0.7,-0.3) and (0.3,-0.3) .. (0,0) ;
\foreach \x in {(0,0),(1,0),(2,0.5),(2,-0.5)}
\fill \x circle (5pt);
\end{tikzpicture}}(\spa) 
+n^{\downarrow 4}\, M_{\begin{tikzpicture}[scale=0.3]
\draw (3,0) -- (1,0);
\draw (1,0) .. controls (0.7,0.3) and (0.3,0.3) .. (0,0) ;
\draw (1,0) .. controls (0.7,-0.3) and (0.3,-0.3) .. (0,0) ;
\foreach \x in {(0,0),(1,0),(2,0),(3,0)}
\fill \x circle (5pt);
\end{tikzpicture}}(\spa) 
+n^{\downarrow 4}\, M_{\begin{tikzpicture}[scale=0.3]
\draw (0,0.5) -- (0,-0.5) -- (1,0) -- (0,0.5);
\draw (1,0) -- (2,0);
\foreach \x in {(0,0.5),(0,-0.5),(1,0),(2,0)}
\fill \x circle (5pt);
\end{tikzpicture}}(\spa)
+n^{\downarrow 3}\, M_{\begin{tikzpicture}[scale=0.3]
\draw (2,0) -- (0,0);
\draw (1,0) .. controls (0.7,0.3) and (0.3,0.3) .. (0,0) ;
\draw (1,0) .. controls (0.7,-0.3) and (0.3,-0.3) .. (0,0) ;
\foreach \x in {(0,0),(1,0),(2,0)}
\fill \x circle (5pt);
\end{tikzpicture}}(\spa)  
+n^{\downarrow 3}\, M_{\begin{tikzpicture}[scale=0.3]
\draw (0,-0.5) -- (1,0) -- (0,0.5);
\draw (0,0.5) .. controls (0.3,0.2) and (0.3,-0.2) .. (0,-0.5);
\draw (0,0.5) .. controls (-0.3,0.2) and (-0.3,-0.2) .. (0,-0.5);
\foreach \x in {(0,0.5),(0,-0.5),(1,0)}
\fill \x circle (5pt);
\end{tikzpicture}}(\spa) \right) \\
&\quad +6\,n^{\downarrow 3}\, M_{\begin{tikzpicture}[scale=0.3]
\draw (2,0) .. controls (1.7,-0.3) and (1.3,-0.3) .. (1,0) .. controls (0.7,0.3) and (0.3,0.3) .. (0,0) ;
\draw (2,0) .. controls (1.7,0.3) and (1.3,0.3) .. (1,0) .. controls (0.7,-0.3) and (0.3,-0.3) .. (0,0) ;
\foreach \x in {(0,0),(1,0),(2,0)}
\fill \x circle (5pt);
\end{tikzpicture}}(\spa) + 4\left(
n^{\downarrow 5}\,M_{\begin{tikzpicture}[scale=0.3]
\draw (0,0) -- (4,0);
\foreach \x in {(0,0),(1,0),(2,0),(3,0),(4,0)}
\fill \x circle (5pt);
\end{tikzpicture}}(\spa)
+n^{\downarrow 5}\, M_{\begin{tikzpicture}[scale=0.3]
\draw (0,0) -- (2,0) -- (3,0.5);
\draw (2,0) -- (3,-0.5);
\foreach \x in {(0,0),(1,0),(2,0),(3,0.5),(3,-0.5)}
\fill \x circle (5pt);
\end{tikzpicture}}(\spa)
+n^{\downarrow 4}\,M_{\begin{tikzpicture}[scale=0.3]
\draw (3,0) -- (2,0) .. controls (1.7,-0.3) and (1.3,-0.3) .. (1,0) -- (0,0);
\draw (2,0) .. controls (1.7,0.3) and (1.3,0.3) .. (1,0);
\foreach \x in {(0,0),(1,0),(2,0),(3,0)}
\fill \x circle (5pt);
\end{tikzpicture}}(\spa)
\right)\\
&\quad + 2\left(
n^{\downarrow 5}\,M_{\begin{tikzpicture}[scale=0.3]
\draw (0,0.5) -- (2,0.5);
\draw (1,-0.5) .. controls (0.7,-0.8) and (0.3,-0.8) .. (0,-0.5);
\draw (1,-0.5) .. controls (0.7,-0.2) and (0.3,-0.2) .. (0,-0.5);
\foreach \x in {(0,0.5),(1,0.5),(2,0.5),(0,-0.5),(1,-0.5)}
\fill \x circle (5pt);
\end{tikzpicture}}(\spa)
+ n^{\downarrow 4}\,M_{\begin{tikzpicture}[scale=0.3]
\draw (0,-0.5) rectangle (1,0.5);
\foreach \x in {(0,0.5),(1,0.5),(1,-0.5),(0,-0.5)}
\fill \x circle (5pt);
\end{tikzpicture}}(\spa)
+n^{\downarrow 2}\, M_{\begin{tikzpicture}[scale=0.3]
\draw (0,0) .. controls (0.5,0.2) and (1,0.2) .. (1.5,0);
\draw (0,0) .. controls (0.5,-0.2) and (1,-0.2) .. (1.5,0);
\draw (0,0) .. controls (0.5,0.6) and (1,0.6) .. (1.5,0);
\draw (0,0) .. controls (0.5,-0.6) and (1,-0.6) .. (1.5,0);
\foreach \x in {(0,0),(1.5,0)}
\fill \x circle (5pt);
\end{tikzpicture}}(\spa)
\right)\\
&\quad + n^{\downarrow 6}\,M_{\begin{tikzpicture}[scale=0.3]
\draw (0,0.5) -- (2,0.5);
\draw (0,-0.5) -- (2,-0.5);
\foreach \x in {(0,0.5),(1,0.5),(2,0.5),(0,-0.5),(1,-0.5),(2,-0.5)}
\fill \x circle (5pt);
\end{tikzpicture}}(\spa)
+ n^{\downarrow 5}\,M_{\begin{tikzpicture}[scale=0.3]
\draw (0,0) -- (1.5,0);
\draw (0.75,0.75) -- (0.75,-0.75);
\foreach \x in {(0,0),(1.5,0),(0.75,0),(0.75,0.75),(0.75,-0.75)}
\fill \x circle (5pt);
\end{tikzpicture}}(\spa)
+ n^{\downarrow 4}\,M_{\begin{tikzpicture}[scale=0.3]
\draw (1,-0.5) .. controls (0.7,-0.8) and (0.3,-0.8) .. (0,-0.5);
\draw (1,-0.5) .. controls (0.7,-0.2) and (0.3,-0.2) .. (0,-0.5);
\draw (1,0.5) .. controls (0.7,0.8) and (0.3,0.8) .. (0,0.5);
\draw (1,0.5) .. controls (0.7,0.2) and (0.3,0.2) .. (0,0.5);
\foreach \x in {(0,0.5),(1,0.5),(0,-0.5),(1,-0.5)}
\fill \x circle (5pt);
\end{tikzpicture}}(\spa).
\end{align*}
\end{exa}

\subsection{The three first limiting cumulants} 
Proposition \ref{prop:moments_monomial_observables} shows that if one can compute $M_G(\spa)$ for any graph $G$, then one can also compute the moments and cumulants of $M_{G}(\spa_n)$ for any $n$ and any graph $G$. However, even in the easy case where $\spa$ is the circle, it can be difficult to find the value of the integral 
\begin{align*}
M_G(\Tor) &= \int_{[0,1]^{p}} \left(\prod_{\{a,b\} \in E(G)} d(\overline{x}_a,\overline{x}_b)\right) dx_1\cdots dx_p \\
&= \int_{[0,1]^{p}} \left(\prod_{\{a,b\} \in E(G)} \min(|x_a-x_b|,|x_a-x_b+1|,|x_a-x_b-1|)\right) dx_1\cdots dx_p.
\end{align*}
In the following, we compute the three first moments of $\Phi(\spa_n) = M_{\begin{tikzpicture}[scale=0.3]
\draw (0,0) -- (2,0);
\foreach \x in {(0,0),(1,0),(2,0)}
\fill \x circle (5pt);
\end{tikzpicture}}(\spa_n)$, and we explain in the specific case where $\mathcal{X} = \R/\Z =\Tor$ how to make some \emph{reductions} of the graphs $G$ that appear in our computation.
\bigskip

We have of course $M_{\begin{tikzpicture}[scale=0.3]
\fill (0,0) circle (5pt);
\end{tikzpicture}}(\Tor)=1$. Let us explain how to compute $M_G(\Tor)$ when one can reduce $G$ to the trivial graph $\bullet$  by recursively deleting in $G$ the vertices with one or two neighbors:

\begin{itemize}
	\item reduction of the vertices with one neighbor. If in the graph $G$ there is one vertex $x$ only connected to another vertex $y$, then we can factor in the integral $M_G(\Tor)$ the term
	$$\int_{\Tor} (d(x,y))^a \,dx,$$
	where $a\geq 1$ is the number of edges between $x$ and $y$. The integral above is equal to
	$$2\,\int_0^{\frac{1}{2}} t^a\,dt = \frac{1}{2^{a}(a+1)}.$$
	Therefore,
	$$M_{\begin{tikzpicture}[scale=0.3]
	\draw (0.5,0) -- (2.5,0);
	\draw (1.5,0.35) node {\tiny $a$};
	\fill (2.5,0) circle (5pt);
	\draw (0,0) node {$G$}; 
	\end{tikzpicture}}(\Tor) = \frac{1}{2^a(a+1)}\,M_G(\Tor).$$
	More generally, because the circle $\Tor$ is a homogeneous space, if the graph $G$ is not biconnected and can be written either as the disjoint union of two graphs $G_1$ and $G_2$, or as the union of two graphs $G_1$ and $G_2$ that only spare one vertex, then we have $M_G(\Tor) = M_{G_1}(\Tor)\,M_{G_2}(\Tor)$.

	\item reduction of the vertices with two neighbors. Suppose now that there is one vertex $x$ only connected to two other vertices $y$ and $z$, with $a\geq 1$ edges between $x$ and $y$ and $b \geq 1$ edges between $x$ and $z$. Note that this does not mean that one can split $G$ as the union of two biconnected components meeting at $x$ (consider for instance the case where $y$ and $z$ are themselves connected by an edge). We have
	\begin{align*}
	\int_{\Tor} (d(x,y))^a\,(d(x,z))^b\, dx&= \int_0^{D} t^a\,(D-t)^b\,dt + \int_0^{\frac{1}{2}-D} t^a\,(D+t)^b\,dt \\
	&\quad+ \int_0^{\frac{1}{2}-D} (D+t)^a\,t^b\,dt +\int_{\frac{1}{2}-D}^{\frac{1}{2}} t^a\,(1-D-t)^b\,dt 
	\end{align*}
	with $D = d(y,z)$. These four terms are polynomials in $D$:
	\begin{align*}
	\int_0^{D} t^a\,(D-t)^b\,dt &= \frac{a!\,b!}{(a+b+1)!}\,D^{a+b+1} ;\\
	\int_0^{\frac{1}{2}-D} t^a\,(D+t)^b\,dt &= \sum_{j=0}^{a+b+1} \left(\sum_{k=0}^{\min(b,j)} \binom{b}{k}\binom{a+b+1-k}{a+b+1-j}\,\frac{(-1)^{j-k}}{2^{a+b+1-j}(a+b+1-k)} \right)D^{j}  ;\end{align*}
	\begin{align*}
	\int_0^{\frac{1}{2}-D} (D+t)^a\,t^b\,dt &= \sum_{j=0}^{a+b+1} \left(\sum_{k=0}^{\min(a,j)} \binom{a}{k}\binom{a+b+1-k}{a+b+1-j}\,\frac{(-1)^{j-k}}{2^{a+b+1-j}(a+b+1-k)} \right)D^{j} ;\\
	 \int_{\frac{1}{2}-D}^{\frac{1}{2}} t^a\,(1-D-t)^b\,dt  
	 &= \frac{1}{2^{a+b}} \sum_{j=0}^{a+b} \left(\sum_{\substack{0\leq k \leq j \\ k\text{ even}}}\sum_{l=0}^k   \binom{a}{k-l}\binom{b}{l}\binom{a+b-k}{a+b-j}\frac{(-1)^{j+l}}{k+1}\right)\,D^{j+1}. 
	\end{align*}
	Therefore, if a graph $G$ contains a vertex $x$ with $a$ incident edges $(x,y)$, $b$ incident edges $(x,z)$ and no other incident edges, then we have the reduction formula
	\begin{align*}
	M_G(\Tor) &= \frac{a!\,b!}{(a+b+1)!}\,M_{(G \setminus \{x\}) + (y,z)^{a+b+1}}(\Tor) \\
	&\quad+ \sum_{\substack{0\leq k \leq j \leq a+b+1}} \binom{a}{k} \binom{a+b+1-k}{a+b+1-j}\,\frac{(-1)^{j+k}}{2^{a+b+1-j}(a+b+1-k)} \,M_{(G \setminus \{x\}) + (y,z)^j}(\Tor)\\
	&\quad+ \sum_{\substack{0\leq k \leq j \leq a+b+1}} \binom{b}{k}\binom{a+b+1-k}{a+b+1-j}\,\frac{(-1)^{j+k}}{2^{a+b+1-j}(a+b+1-k)} \,M_{(G \setminus \{x\}) + (y,z)^j}(\Tor)\\
	&\quad + \frac{1}{2^{a+b}} \sum_{\substack{0\leq l \leq k\leq j \leq a+b \\ k \text{ even}}}  \binom{a}{k-l}\binom{b}{l}\binom{a+b-k}{a+b-j}\frac{(-1)^{j+l}}{k+1}\,M_{(G \setminus \{x\}) + (y,z)^{j+1}}(\Tor),
	\end{align*}
	where $(G \setminus \{x\}) + (y,z)^j$ is the graph obtained from $G$ by first removing the vertex $x$ and its adjacent edges, and then adding $j$ new edges between $y$ and $z$.	
\end{itemize}
This is already sufficient in order to compute the two first moments of $\Phi(\spa_n)$:
\begin{align*}
n^3\,\esper[\Phi(\spa_n)] &= n^{\downarrow 3}\,M_{\begin{tikzpicture}[scale=0.3]
\draw (0,0) -- (2,0);
\foreach \x in {(0,0),(1,0),(2,0)}
\fill \x circle (5pt);
\end{tikzpicture}}(\Tor) + n^{\downarrow 2}\,M_{\begin{tikzpicture}[scale=0.3]
\draw (1,0) .. controls (0.7,0.3) and (0.3,0.3) .. (0,0) ;
\draw (1,0) .. controls (0.7,-0.3) and (0.3,-0.3) .. (0,0) ;
\foreach \x in {(0,0),(1,0)}
\fill \x circle (5pt);
\end{tikzpicture}}(\Tor) = \frac{n^{\downarrow 3}}{16} + \frac{n^{\downarrow 2}}{12} = \frac{n^3}{16} - \frac{5n^2}{48} + \frac{n}{24};\\
n^6\,\esper[(\Phi(\spa_n))^2] &= 
\frac{n^6}{256} - \frac{5n^5}{384} + \frac{611n^4}{26880} - \frac{67n^3}{2688} + \frac{5n^2}{336} + \frac{n}{280}.
\end{align*}
Indeed, all the loopless graphs $G^r \downarrow \pi$ with $r \in \{1,2\}$ are reducible by one of the previous arguments. We obtain in particular the value of the variance:
$$n^2\,\var(\Phi(\spa_n)) = 
\frac{269}{40320} - \frac{131}{8064n} + \frac{53}{4032 n^2} - \frac{1}{280n^3}.$$
In particular, $\sigma_{\mathrm{hom}}^2 = \frac{269}{40320}$ is strictly positive.\bigskip

For the third moment, there are $22147$ set partitions of size $9$, and $6097$ of them yield a contracted graph $G^3 \downarrow \pi$ which is without loop. These $6097$ graphs fall into $131$ isomorphism types, and only one isomorphism type is not reducible with the aforementioned techniques:
$$H = G^3 \downarrow \pi = K_4 = \begin{tikzpicture}[baseline=3mm,scale=1]
\draw (1,1) -- (0,0) -- (0,1) -- (1,0) -- (1,1) -- (0,1);
\draw (0,0) -- (1,0);
\foreach \x in {(0,0),(0,1),(1,0),(1,1)}
\fill \x circle (2pt);
\end{tikzpicture}\,.$$
Let us explain how to compute $M_H(\Tor)$. We need to compute the integral
$$I = \int_{\Tor}d(w,x)\,d(w,y)\,d(w,z)\,dw.$$
The Fourier expansion of the distance function $d(x,y)$ with $x,y \in \R/\Z$ is
$$d(x,y) = \frac{1}{4} - \sum_{\substack{k \in \Z \\ k \text{ odd}}} \frac{1}{k^2\pi^2}\,\E^{2\I \pi k (x-y)}.$$
If $\widetilde{d}(x,y) = \frac{1}{4}-d(x,y)$, then
\begin{align*}
 I &= -  \int_{\Tor}\left(\widetilde{d}(w,x)-\frac{1}{4}\right) \left(\widetilde{d}(w,y)-\frac{1}{4}\right) \left(\widetilde{d}(w,z)-\frac{1}{4}\right) dw \\ 
 &= - \int_{\Tor} \widetilde{d}(w,x)\,\widetilde{d}(w,y)\,\widetilde{d}(w,z)\,dw + \frac{1}{4} (F(x,y) + F(x,z) + F(y,z))+ \frac{1}{64}
\end{align*}
where 
\begin{align*}
F(x,y) &= \int_\Tor \widetilde{d}(w,x)\,\widetilde{d}(w,y)\,dw =  \int_\Tor d(w,x)\,d(w,y)\,dw -\frac{1}{16} \\
&= \frac{2(d(x,y))^3}{3} - \frac{(d(x,y))^2}{2} + \frac{1}{48}.
\end{align*}
Now, the key observation is that  $\int_\Tor \widetilde{d}(w,x)\,\widetilde{d}(w,y)\,\widetilde{d}(w,z)\,dw=0$. Indeed, by using the Fourier expansions of the distance functions, setting $C_{k \text{ odd}}=\frac{1}{k^2\pi^2}$, we see that this integral equals
$$\sum_{k+l+m=0} C_kC_lC_m\, \E^{-2\I \pi (kx + ly + mz)}.$$
the sum running over odd integers $k$, $l$ and $m$. But then it is not possible to have $k+l+m=0$, whence the vanishing of the integral. As a consequence,
$$M_{K_4}(\Tor) = \frac{1}{2} \,M_{\begin{tikzpicture}[scale=0.3]
\draw (0,0.5) -- (1.5,0) -- (0,-0.5);
\draw (0,0.5) .. controls (0.2,0.2) and (0.2,-0.2) .. (0,-0.5);
\draw (0,0.5) .. controls (-0.2,0.2) and (-0.2,-0.2) .. (0,-0.5);
\draw (0,0.5) .. controls (0.6,0.2) and (0.6,-0.2) .. (0,-0.5);
\draw (0,0.5) .. controls (-0.6,0.2) and (-0.6,-0.2) .. (0,-0.5);
\foreach \x in {(0,0.5),(0,-0.5),(1.5,0)}
\fill \x circle (5pt);
\end{tikzpicture}}(\Tor) - \frac{3}{8}\,M_{\begin{tikzpicture}[scale=0.3]
\draw (0,0.5) -- (1.5,0) -- (0,-0.5);
\draw (0,0.5) .. controls (0.2,0.2) and (0.2,-0.2) .. (0,-0.5);
\draw (0,0.5) .. controls (-0.2,0.2) and (-0.2,-0.2) .. (0,-0.5);
\draw (0,0.5) .. controls (-0.6,0.2) and (-0.6,-0.2) .. (0,-0.5);
\foreach \x in {(0,0.5),(0,-0.5),(1.5,0)}
\fill \x circle (5pt);
\end{tikzpicture}}(\Tor) + \frac{1}{32} M_{\begin{tikzpicture}[scale=0.3]
\draw (0,0.5) -- (1.5,0) -- (0,-0.5) -- (0,0.5);
\foreach \x in {(0,0.5),(0,-0.5),(1.5,0)}
\fill \x circle (5pt);
\end{tikzpicture}}(\Tor)= \frac{11}{71680},$$
all the graphs on the right-hand side being reducible. By using a computer algebra system, we then obtain
\begin{align*}
n^9\,\esper[(\Phi(\spa_n))^3] &=  \frac{n^9}{4096} - \frac{5n^8}{4096} + \frac{541n^7}{143360} - \frac{5713619n^6}{638668800}+ \frac{61771n^5}{3801600}  \\
&\quad - \frac{132443n^4}{6386688} + \frac{6367n^3}{380160} - \frac{150193n^2}{19958400} + \frac{2353n}{1663200}.
\end{align*}
This gives the third cumulant:
$$ n^3\,\kappa^{(3)}(\Phi(\spa_n)) = -\frac{42209}{39916800} + O\!\left(\frac{1}{n}\right).$$
Since the right-hand side does not vanish, we conclude that $\lim_{n \to \infty} \kappa^{(3)}(Y_{n}(\varphi,\spa)) \neq 0 $; therefore, the limiting distribution from Theorem \ref{theo:homogeneous} is not the standard normal distribution, and we have proved:
\begin{prop}
If $\spa$ is the circle  $\R/\Z$ endowed with its geodesic distance and with the projection of the Lebesgue measure, and if $\Phi=M_{\begin{tikzpicture}[scale=0.3]
\draw (0,0) -- (2,0);
\foreach \x in {(0,0),(1,0),(2,0)}
\fill \x circle (5pt);
\end{tikzpicture}}$, then the Gromov--Prohorov sample model yields a sequence of random variables
$$\frac{\Phi(\spa_n)-\Phi(\spa)}{\sqrt{\var(\Phi(\spa_n))}}$$
which converges in distribution towards a law which is centered, with variance $1$ and with third cumulant equal to
$$-\frac{168836}{44385}\sqrt{\frac{70}{269}}\simeq -1.94044;$$
in particular, this distribution is not the Gaussian distribution.
\end{prop}
\bigskip

\section{A statistical test for the symmetry of a compact Riemannian manifold}\label{sec:statistics}

In this section, as an application of our results and in particular of the concentration inequality \ref{prop:chernoff}, we construct a statistical test for the symmetry of a compact manifold.
\medskip

\noindent \textbf{Model.} We consider a compact Riemannian manifold $\mathcal{X}$; the distance between points of $\mathcal{X}$ is the geodesic distance, and the compactness ensures that for any pair of points $(x,y)$ in $\mathcal{X}$, there exists a geodesic curve of minimal length connecting $x$ to $y$ (see for instance \cite[Section 1.5]{Jost11}; this is even true for complete Riemannian manifolds, see \cite[Chapter I, Theorem 10.4]{Hel78}). The space $X$ is equipped with the probability measure $\mu$ proportional to the Lebesgue measure induced by the Riemannian volume form $\omega$ of the manifold. An isometry of $\mathcal{X}$ always preserves the Riemannian structure and therefore the probability measure $\mu$ (this result is due to Myers and Steenrod \cite{MS39}; see also \cite[Chapter I, Theorem 11.1]{Hel78}). In other words, $\mathrm{Isomp}(\spa)=\mathrm{Isom}(\spa)$. This group of isometries endowed with the compact-open topology is always a compact Lie group, such that the action $G \times \mathcal{X} \to \mathcal{X}$ is a smooth map; see \cite[Chapter II, Theorems 1.1 and 1.2]{Kob72}. Therefore, the following assertions are equivalent:
\begin{enumerate}
 	\item The Riemannian manifold $\mathcal{X}$ is a compact homogeneous space (in the sense of the fourth item of Theorem \ref{theo:homogeneous}).
 	\item The group of isometries $G=\mathrm{Isom}(\spa)$ acts transitively on $\mathcal{X}$.
\end{enumerate}  
Our objective is to construct a statistical test for these conditions.
\medskip

\noindent \textbf{Hypotheses and statistics.} The hypotheses of our test are:
\begin{align*}
\text{null hypothesis }H_0&: \,\,\,\text{the compact manifold $\spa$ is homogeneous};\\
\text{alternative hypothesis }H_1&:\,\,\,\text{the compact manifold $\spa$ is not homogeneous}.
\end{align*}
The allowed observations of $\spa$ are the following:
\begin{itemize}
 	\item we can take independent random points $x_1,x_2,\ldots,x_n$ on $\mathcal{X}$, all these points being chosen according to the Lebesgue measure $\mu$;
 	\item  and we can measure all their inter-distances $d(x_i,x_j)$, $1 \leq i,j \leq n$.
 \end{itemize}  Let us fix a function $\varphi \in \mathcal{C}_b(\R^{\binom{p}{2}})$ corresponding to a polynomial observable $\Phi=\Phi^{p,\varphi}$ of mm-spaces. A convenient choice is
$$p=2\quad;\quad\varphi(d) = \min(1,d),$$
but other observables might yield more powerful tests; we shall discuss this in a moment. \medskip

By Theorem \ref{theo:singular_case}, the random variable $\Phi(\spa_n) = \frac{1}{n^p} \sum_{i_1,\ldots,i_p = 1}^n \varphi(d(x_{i_1},\ldots,x_{i_p}))$ has fluctuations of order $\frac{1}{n}$ under the hypothesis $H_0$, so a convenient statistics for testing this hypothesis would be $n\,|\Phi(\spa_n)-\Phi(\spa)|$. As we do not know the value of $\Phi(\spa)$, we shall proceed a bit differently. Consider an independent copy $\spa_n'$ of the discrete approximation of our mm-space $\spa$, constructed from random points $x_1',x_2',\ldots,x_n'$ which are again independent, independent from $x_1,\ldots,x_n$, and distributed on $\mathcal{X}$ according to the normalised Lebesgue measure $\mu$. By the triangular inequality, the statistics
\begin{equation}
	Z_n = n\,|\Phi(\spa_n)-\Phi(\spa_n')| \label{eq:statistics}
\end{equation}
is smaller than the sum of two independent random variables distributed like $n\,|\Phi(\spa_n)-\Phi(\spa)|$, so given a large threshold $t_\alpha$, we should again have $Z_n \leq t_\alpha$ with large probability under $H_0$. We therefore choose $Z_n$ as our statistics of test.
\medskip

\noindent \textbf{Estimates of probabilities and choice of the threshold.}
By Proposition \ref{prop:chernoff}, under the hypothesis of symmetry $H_0$, if $A$ is an upper bound on $\|\varphi\|_\infty$ ($A=1$ if we consider the test function $\varphi(d) = \min(1,d)$), then
$$\proba_{H_0}\!\left[n\,|\Phi(\spa_n)-\Phi(\spa)| \geq \frac{2Ap^2}{\E}\,x\right] \leq 2\,\exp\!\left(\frac{\log(1+x)-x}{\E^2}\right).$$
Notice that this is a non-asymptotic estimate, valid for any $n \geq 1$. It implies:
\begin{align*}
&\proba_{H_0}\!\left[Z_n \geq \frac{4Ap^2}{\E}\,x\right]\\
 &= \proba_{H_0}\!\left[n\,|\Phi(\spa_n)-\Phi(\spa_n')| \geq \frac{4Ap^2}{\E}\,x\right]\\
&\leq \proba_{H_0}\!\left[n\,|\Phi(\spa_n)-\Phi(\spa)| \geq \frac{2Ap^2}{\E}\,x\right] + \proba_{H_0}\!\left[n\,|\Phi(\spa_n')-\Phi(\spa)| \geq \frac{2Ap^2}{\E}\,x\right]\\
&\leq 4\,\exp\!\left(\frac{\log(1+x)-x}{\E^2}\right)=F(x).
\end{align*}
The upper bound $F(x)$ is a strictly decreasing function of $x$ with $\lim_{x \to \infty} F(x) = 0$. Therefore, for every significance level $\alpha \in (0,1)$, there exists a unique $x_\alpha \in \R_+$ with $F(x_\alpha)=\alpha$.
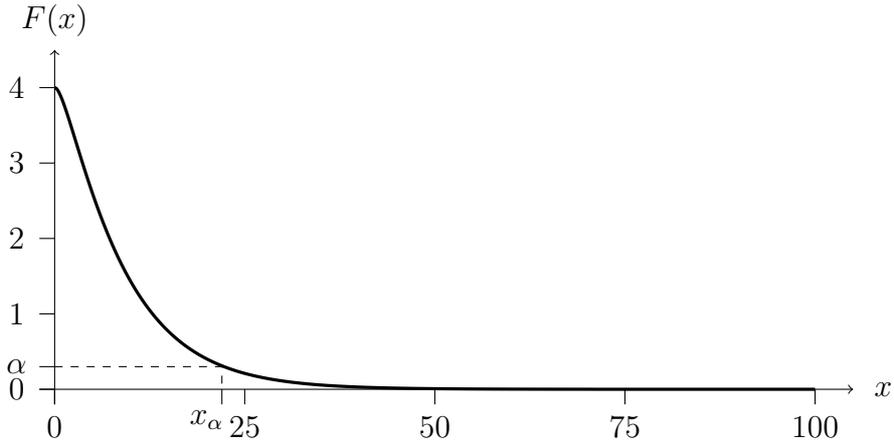
 \begin{center}
 \begin{tikzpicture}[xscale=0.1]
 \draw [smooth, domain=0:100, samples=500, very thick] plot (\x,{4*exp((ln(1+\x)-\x)/exp(2))});
 \draw [->] (-2,0) -- (105,0);
 \draw [->] (0,-0.2) -- (0,4.5);
 \foreach \x in {0,1,2,3,4}
 {\draw (-5,\x) node {$\x$};
 \draw (-2,\x) -- (0,\x);}
 \foreach \x in {0,25,50,75,100}
 {\draw (\x,-0.5) node {$\x$};
 \draw (\x,-0.2) -- (\x,0);}
 \draw (0,4.9) node {$F(x)$};
 \draw (109,0) node {$x$};
 \draw (-5,0.3) node {$\alpha$};
 \draw (-2,0.3) -- (0,0.3);
 \draw [dashed] (0,0.3) -- (22,0.3) -- (22,0);
 \draw (22,0) -- (22,-0.2);
 \draw (20,-0.4) node {$x_\alpha$};
 \end{tikzpicture}
 \end{center}
We set
\begin{equation}
t_\alpha = \frac{4Ap^2}{\E}\,x_\alpha = \frac{4Ap^2}{\E}\,F^{-1}(\alpha).\label{eq:threshold}
\end{equation}
\begin{prop}\label{prop:test_symmetry}
 For any $n \geq 1$ and any significance level $\alpha \in (0,1)$, the statistics $Z_n$ given by Equation \eqref{eq:statistics} and the threshold $t_\alpha$ given by Equation \eqref{eq:threshold} yield a test for the hypothesis of symmetry with level smaller than $\alpha$:
 $$\proba_{H_0}[Z_n \geq t_\alpha] \leq \alpha.$$
\end{prop}
\medskip

\noindent \textbf{Power of the test and sample size.} Of course, the proposition above is only useful if we can also estimate the probability $\proba_{H_1}[Z_n < t_\alpha]$ of the second kind of error of this procedure of testing, and make it reasonably small. 
To this purpose, we need to be a bit more precise on the alternative hypothesis $H_1$ (if $\spa$ is very close to being homogeneous, then the probability of the second kind of error will be large). A typical example which is solvable is the following. Suppose that we observe a manifold $\spa=(\mathcal{X},d,\mu)$ isometric to the circle $\R/\Z$, and where $\frac{d\mu}{dx}=f(x)$ is some unknown density function with $\int_0^1 f(x)\,dx=1$. 
In this case, the hypotheses of our test for symmetry can be taken as follows:
\begin{align*}
H_0&:\,\,\,f \text{ is constant (and the space is homogeneous)};\\
H_1(\varepsilon)&:\,\,\,f(x)\,dx \text{ is at total variation distance larger than $\varepsilon$ from }dx
\end{align*}
for some $\varepsilon>0$.\bigskip

More generally, we can take for alternative hypothesis:
\begin{align*}
H_1':\,\,\,&\text{$\spa$ belongs to a specific class of non-homogeneous compact manifolds}\\[-1mm]
&\text{for which $\sigma^2(\varphi,\spa)$ can be computed}.
\end{align*}
Then, we can follow the steps below in order to compute the power of our test:
 
 \begin{enumerate}
 \item Compute a lower bound $\sigma^2_0$ on $\sigma^2(\varphi,\spa)$ for $\spa$ described by $H_1'$. We assume that this lower bound is strictly positive (under $H_0$, $\sigma^2(\varphi,\spa)=0$ for any continuous bounded function $\varphi$).

 \item By Theorem \ref{theo:generic_case}, $\sqrt{n}\,(\Phi(\spa_n) - \Phi(\spa))$ converges to a centered normal distribution with variance $p^2\,\sigma^2(\varphi,\spa)$, so
 $$\frac{Z_n}{\sqrt{n}} \rightharpoonup_{n \to \infty} |\mathcal{N}(0,2p^2\,\sigma^2(\varphi,\spa))|.$$
 Moreover, the Kolmogorov distance between these two distributions is a 
 $$O\!\left(\frac{A^3p}{\sigma^3(\varphi,\spa)\,\sqrt{n}}\right),$$
 with a universal constant $C$ in the $O(\cdot)$ (an explicit value of $C$ can be computed readily from \cite[Corollary 30]{feray2017mod}). Therefore,
 \begin{align*}
 \proba[Z_n < t_\alpha] &= \proba\!\left[\frac{Z_n}{\sqrt{n}} \leq \frac{t_\alpha}{\sqrt{n}}\right] \\
 &\leq \frac{CA^3p}{\sigma^3(\varphi,\spa)\,\sqrt{n}} + \proba\left[|\mathcal{N}(0,2p^2\,\sigma^2(\varphi,\spa))|\leq \frac{t_\alpha}{\sqrt{n}}\right]\\
 &\leq \frac{CA^3p}{\sigma^3(\varphi,\spa)\,\sqrt{n}} + \frac{C'A p\,F^{-1}(\alpha) }{\sigma(\varphi,\spa)\,\sqrt{n}}
 \end{align*}
 with $C'=\frac{4}{\E\sqrt{\pi}}$. 
 
 \item By combining the two items above, we obtain:
 $$\proba_{H_1'}[Z_n < t_\alpha] \leq \left(\frac{A^3}{\sigma_0^3} + \frac{A}{\sigma_0}\,F^{-1}(\alpha)\right)\frac{Kp}{\sqrt{n}}$$
 for some universal constant $K$.
 \end{enumerate}

So, we conclude:
\begin{prop}
Fix a significance level $\alpha$ and a threshold $t_\alpha$ as in Proposition \ref{prop:test_symmetry}. There exists a universal constant $K$ such that the test for symmetry $H_0/H_1'$ has power larger than $1-\beta$, with
$$\beta =  \left(\frac{A^3}{\sigma_0^3} + \frac{A}{\sigma_0}\,F^{-1}(\alpha)\right)\frac{Kp}{\sqrt{n}},$$
and where $A=\|\varphi\|_\infty$ and $\sigma_0^2$ is a lower bound on $\sigma^2(\varphi,\spa)$ under the alternative hypothesis $H_1'$.
\end{prop}
\noindent Therefore, once the observable $\varphi$ and the significance level $\alpha$ of the test for symmetry have been chosen, if one has a non-zero lower bound on the variances $\sigma^2(\varphi,\spa)$ under $H_1'$, then one can find a sample size $n$ in order to obtain an error of the second kind as small as wanted. 

\begin{remark}
Suppose that $p=2$ and that $\varphi(d) = \min(A,d)$ (standard choice of observable), where $A$ is an \emph{a priori} upper bound on the diameter of the space $\spa$ to which we want to apply the test for symmetry. Then,
$$\sigma^2(\varphi,\spa) = \int_{\spa^3} d(x,y)d(y,z)\, \mu^{\otimes 3}(dx\,dy\,dz) - \left(\int_{\spa^2}d(x,y)\,\mu^{\otimes 2}(dx\,dy)\right)^2.$$

\end{remark}

\begin{remark}
One might need to choose $\varphi$ and the observable $\Phi^{p,\varphi}$ in order to ensure that one has under $H_1'$ a non-zero lower bound $\sigma_0^2$. Indeed, for a given non-homogeneous space $\spa$, certain functions $\varphi$ might give "by chance" a vanishing parameter $\sigma^2(\varphi,\spa)$. Consequently, one might have to take another observable than the one previously presented as the standard choice.
\end{remark}

\bigskip

\bibliographystyle{alpha}
\bibliography{references}

\end{document}